\documentclass{amsart}

\usepackage{amsmath}
\usepackage{amssymb}
\usepackage{amscd}
\usepackage{enumerate}
\usepackage{verbatim}
\usepackage{color}
\usepackage{bm}

\input xy
\xyoption{all}

\def\boxit#1{\vbox{\hrule\hbox{\vrule\kern3pt
     \vbox{\kern3pt#1\kern3pt}\kern3pt\vrule}\hrule}}

\usepackage[makeroom]{cancel}

\newtheorem{Thm}{Theorem}[section]
\newtheorem{Lemma}[Thm]{Lemma}
\newtheorem{Cor}[Thm]{Corollary}
\newtheorem{Prop}[Thm]{Proposition}
\theoremstyle{definition}
\newtheorem{Def}[Thm]{Definition}
\newtheorem{Example}[Thm]{Example}

\newtheorem{Rmk}[Thm]{Remark}

\def\frakS{\mathfrak{S}}

\def\bbc{\mathbb{C}}

\def\bbf{\mathbb{F}}

\def\bbn{\mathbb{N}}

\def\bbq{\mathbb{Q}}
\def\bbr{\mathbb{R}}

\def\bbz{\mathbb{Z}}

\def\calA{\mathcal{A}}

\def\calH{\mathcal{H}}
\def\calI{\mathcal{I}}

\def\calO{\mathcal{O}}

\def\calR{\mathcal{R}}
\def\calS{\mathcal{S}}

\def\lra{\longrightarrow}

\def\x{\times}
\def\bs{\backslash}

\def\aut{\mathrm{Aut}}
\def\aff{\mathrm{Aff}}
\def\fix{\mathrm{Fix}}

\def\tr{\mathrm{tr\,}}

\def\ad{\mathrm{ad}}
\def\sp{\mathrm{sp}}

\def\frakS{\mathfrak{S}}

\def\sgn{\mathrm{sign~\!}}

\def\R{\mathrm{(R)}}

\def\HPer{\mathrm{HPer}}

\def\DH{\mathrm{DH}}

\def\lcm{\mathrm{LCM}}

\def\bb{|\hspace{-1pt}|}

\def\Per{\mathrm{Per}}
\def\EP{\mathrm{EP}}
\def\DA{\mathrm{DA}}

\def\NP{\mathrm{NP}}
\def\NF{\mathrm{NF}}

\def\alg{\mathrm{alg}}

\begin{document}

\title[The Nielsen and Reidemeister theories of iterations]
{The Nielsen and Reidemeister theories of iterations
on infra-solvmanifolds of type $\R$ and poly-Bieberbach groups}

\author{Alexander Fel'shtyn}
\address{Instytut Matematyki, Uniwersytet Szczecinski,
ul. Wielkopolska 15, 70-451 Szczecin, Poland}
\email{fels@wmf.univ.szczecin.pl}

\author{Jong Bum Lee}
\address{Department of mathematics, Sogang University, Seoul 121-742, KOREA}
\email{jlee@sogang.ac.kr}
\thanks{The second author is supported in part by Basic Science Researcher Program
through the National Research Foundation of Korea(NRF) funded by the Ministry of Education
(No.~\!2013R1A1A2058693) and by the Sogang University Research Grant of 2010 (10022).}

\subjclass[2000]{Primary 37C25; Secondary 55M20}
\date{XXX 1, 2014 and, in revised form, YYY 22, 2015.}

\keywords{Infra-solvmanifold, Nielsen number, Nielsen zeta function, periodic $[\varphi]$-orbit,
poly-Bieberbach group, Reidemeister number, Reidemeister zeta function}

\begin{abstract}
We study the asymptotic behavior of the sequence of the Nielsen numbers $\{N(f^k)\}$,
the essential periodic orbits of $f$ and the homotopy minimal periods of $f$
by using the Nielsen theory of maps $f$ on infra-solvmanifolds of type $\R$.
We develop the Reidemeister theory for the iterations of any endomorphism $\varphi$
on an arbitrary group and study the asymptotic behavior of the sequence of
the Reidemeister numbers $\{R(\varphi^k)\}$, the essential periodic $[\varphi]$-orbits
and the heights of $\varphi$ on poly-Bieberbach groups.
\end{abstract}

\maketitle


\section{Introduction}\label{introd}

Let $f:X\to X$ be a map on a connected compact polyhedron $X$.
A point $x\in X$ is a fixed point of $f$ if $f(x)=x$ 
and is a periodic point of $f$ with period $n$ if $f^n(x)=x$.
The smallest period of a periodic $x$ is called the {\bf minimal period}.
We will use the following notations:
\begin{align*}
&\fix(f)=\{x\in X\mid f(x)=x\},\\
&\Per(f)=\text{the set of all minimal periods of $f$},\\
&P_n(f)=\text{the set of all periodic points of $f$ with minimal period $n$},\\
&{\HPer(f)=\bigcap_{g\simeq f}\left\{n\in\bbn\mid P_n(g)\ne\emptyset\right\}}\\
&\hspace{1.4cm}=\text{the set of all {\bf homotopy minimal periods} of $f$}.
\end{align*}
Let $p:\tilde{X}\rightarrow X$ be the universal covering projection onto $X$
and $\tilde{f}:\tilde{X}\rightarrow \tilde{X}$ a fixed lift of $f$.
Let $\Pi$ be the group of covering transformations of the projection $p:\tilde{X}\to X$.
Then $f$ induces an endomorphism $\varphi=\varphi_f:\Pi\to\Pi$ by the following identity
$$
\varphi(\alpha)\tilde{f}=\tilde{f}\alpha,\quad\forall\alpha\in\Pi.
$$
The subsets $p(\fix(\alpha\tilde{f}))\subset \fix(f)$, $\alpha\in\Pi$,  are called {\bf fixed point classes} of $f$.
A fixed point class is called {\bf essential} if its index is nonzero.
The number of essential fixed point classes is called the {\bf Nielsen number} of $f$, denoted by $N(f)$ \cite{Jiang}.

The Nielsen number is always finite and is a homotopy invariant lower bound for the number of fixed points of $f$.
In the category of compact, connected polyhedra the Nielsen number of a map is,
apart from in certain exceptional cases, equal to the least number of fixed points of maps
with the same homotopy type as $f$.

Let $\varphi:\Pi\to\Pi$ be an endomorphism on an arbitrary group $\Pi$.
Consider  the {\bf Reidemeister action} of $\Pi$ on $\Pi$ determined by the endomorphism $\varphi$
and defined as follows:
$$
\Pi\x\Pi\lra\Pi,\quad (\gamma,\alpha)\mapsto\gamma\alpha\varphi(\gamma)^{-1}.
$$
The Reidemeister class containing $\alpha$ will be denoted by $[\alpha]$,
and the set of Reidemeister classes of $\Pi$ determined by $\varphi$
will be denoted by $\calR[\varphi]$.
Write $R(\varphi)=\#\calR[\varphi]$, called the {\bf Reidemeister number} of $\varphi$.
When the endomorphism $\varphi:\Pi\to\Pi$ is induced from a self-map $f:X\to X$,
i.e., when $\varphi=\varphi_f$, we also refer to $\calR[\varphi]$
as the set $\calR[f]$ of Reidemeister classes of $f$,
and $R(\varphi)$ as the Reidemeister number $R(f)$ of $f$.

It is easy to observe that if $\psi$ is an automorphism on $\Pi$,
then $\psi$ sends the Reidemeister class $[\alpha]$ of $\varphi$
to the Reidemeister class $[\psi(\alpha)]$ of $\psi\varphi\psi^{-1}$.
Hence the Reidemeister number is an automorphism invariant.
For any $\beta\in\Pi$, let $\tau_\beta$ denote the inner automorphism determined by $\beta$.
We will compare $\calR[\varphi]$ with $\calR[\tau_\beta\varphi]$.
Observe that the right multiplication $r_{\beta^{-1}}$ by $\beta^{-1}$ on $\Pi$
induces a bijection $\calR[\varphi]\to\calR[\tau_\beta\varphi]$, $[\alpha]\mapsto[\alpha\beta^{-1}]$.
Indeed,
\begin{align*}
r_{\beta^{-1}}:\gamma\cdot\alpha\cdot\varphi(\gamma)^{-1}\buildrel{}\over\longmapsto
&\ (\gamma\cdot\alpha\cdot\varphi(\gamma)^{-1})\beta^{-1}\\
&= \gamma\cdot(\alpha\beta^{-1})\cdot\beta\varphi(\gamma)^{-1}\beta^{-1}
=\gamma\cdot(\alpha\beta^{-1})\cdot(\tau_\beta\varphi)(\gamma)^{-1}.
\end{align*}
Similarly, we can show that $r_{(\beta\varphi(\beta)\cdots\varphi^{n-1}(\beta))^{-1}}$
induces a bijection $\calR[\varphi^n]\to \calR[(\tau_\beta\varphi)^n]$,
$[\alpha]^n\mapsto [\alpha(\beta\varphi(\beta)\cdots\varphi^{n-1}(\beta))^{-1}]^n$.
Hence the Reidemeister number is a conjugacy invariant.
This is not surprising because if $f$ and $g$ are homotopic,
then their induced endomorphisms differ by an inner automorphism $\tau_\beta$.

The set $\fix(f^n)$ of periodic points of $f$ splits into a disjoint union of periodic point classes
$p(\fix(\alpha\tilde{f}^n))$ of $f$, and these sets are indexed by the Reidemeister classes
$[\alpha]^n\in\calR[\varphi^n]$ of the endomorphism $\varphi^n$ where $\varphi=\varphi_f$.
Namely,
\begin{align}\label{decomp}
\fix(f^n)=\coprod_{[\alpha]^n\in\calR[\varphi^n]}p\left(\fix(\alpha\tilde{f}^n)\right).\tag{D}
\end{align}

From the dynamical point of view, it is natural to consider the Nielsen numbers $N(f^k)$
and the Reidemeister numbers $R(f^k)$ of all iterations of $f$ simultaneously.
For example, N. Ivanov \cite{I} introduced the notion of the asymptotic Nielsen number,
measuring the growth of the sequence $N(f^k)$,
and found the basic relation between the topological entropy of $f$
and the asymptotic Nielsen number.
Later on, it was suggested in \cite{Fel84, PilFel85, Fel88, Fel91, Fel00}
to arrange the Nielsen numbers $N(f^k)$, the Reidemeister numbers $R(f^k)$ and $R(\varphi^k)$ of all iterations of $f$ and $\varphi$
into the Nielsen and the Reidemeister zeta functions
\begin{align*}
&N_f(z)=\exp\left(\sum_{k=1}^\infty\frac{N(f^k)}{k}z^k\right),
\end{align*}
\begin{align*}
R_f(z)=\exp\left(\sum_{k=1}^\infty\frac{R(f^k)}{k}z^k\right),\quad R_\varphi(z)=\exp\left(\sum_{k=1}^\infty\frac{R(\varphi^k)}{k}z^k\right).
\end{align*}
The Nielsen and Reidemeister zeta functions are  nonabelian analogues of the Lefschetz zeta function
$$
L_f(z)=\exp\left(\sum_{k=1}^\infty\frac{L(f^k)}{k}z^k\right),
$$
where
\begin{equation*}
 L(f^n) := \sum_{k=0}^{\dim X} (-1)^k \tr\Big[f_{*k}^n:H_k(X;\bbq)\to H_k(X;\bbq)\Big]
\end{equation*}
is the Lefschetz number of the iterate $f^n$ of $f$.

{Nice analytic properties of $N_f(z)$ \cite{Fel00} indicate that the numbers $N(f^k)$, $k\ge1$, are closely interconnected.}
Other manifestations of this are Gauss congruences
\begin{align*}
\sum_{d\mid k}\mu\!\left(\frac{k}{d}\right)N(f^d)\equiv0\mod{k},
\end{align*}
for any $ k>0$,
where $f$ is  a map on an infra-solvmanifold of type $\R$ \cite{FL}. Whenever all $R(f^k)$ are finite,
we also have
\begin{align*}
\sum_{d\mid k}\mu\!\left(\frac{k}{d}\right)R(f^d)=\sum_{d\mid k}\mu\!\left(\frac{k}{d}\right)N(f^d)\equiv0\mod{k}.
\end{align*}
It is  known that the Reidemeister numbers of the iterates of an automorphism $\varphi$ of an {almost polycyclic group} also  satisfy  Gauss congruences \cite{crelle, ft}.

The fundamental invariants of $f$ used in the study of periodic points are
the Lefschetz numbers $L(f^k)$, and their algebraic combinations, the Nielsen numbers $N(f^k)$
and the Nielsen--Jiang periodic numbers $NP_n(f)$ and $N\Phi_n(f)$, and the Reidemeister numbers $R(f^k)$ and $R(\varphi^k)$.

The study of periodic points by using the Lefschetz theory has been done extensively
by many authors in the literature such as \cite{Jiang}, \cite{Dold}, \cite{BaBo}, \cite{JM}, \cite{Matsuoka}.
A natural question is to ask how much information we can get
about the set of essential periodic points of $f$ or about the set of (homotopy) minimal periods of $f$
from the study of the sequence $\{N(f^k)\}$ of the Nielsen numbers of iterations of $f$.
Utilizing the arguments employed mainly in \cite{BaBo} and \cite[Chap.~III]{JM}
for the Lefschetz numbers of iterations, we study the asymptotic behavior of the sequence $\{N(f^k)\}$,
the essential periodic orbits of $f$ and the homotopy minimal periods of $f$
by using the Nielsen theory of maps $f$ on infra-solvmanifolds of type $\R$.
We will give a brief description of the main results in Section~\ref{N_itetate}
whose details and proofs can be found in \cite{FL2}.
From the identity \eqref{decomp}, the Reidemeister theory for the iterations of $f$
is almost parallel to the Nielsen theory of the iterates of $f$.
Motivated from this parallelism, we will develop in Section~\ref{prelim} the Reidemeister theory
for the iterations of any endomorphism $\varphi$ on an arbitrary group $\Pi$.
In this paper, we will study the asymptotic behavior of the sequence $\{R(\varphi^k)\}$,
the essential periodic $[\varphi]$-orbits and the heights of $\varphi$ on poly-Bieberbach groups.
We refer to \cite{FL,FL2} for background to our present work.

\smallskip
\noindent
\textbf{Acknowledgments.}
The first author is indebted to the Max-Planck-Institute for Mathematics(Bonn) and Sogang University(Seoul)
for the support and hospitality and the possibility of the present research during his visits there.
{The authors also would like to thank Thomas Ward for making careful corrections and suggestions to a few expressions.}

\section{Preliminaries}\label{prelim}

Recall that the periodic point set $\fix(f^n)$ splits into a disjoint union of periodic point classes
$$
\fix(f^n)=\bigsqcup_{[\alpha]^n\in\calR[\varphi^n]}p\left(\fix(\alpha\tilde{f}^n)\right).
$$
Consequently, there is a $1$-$1$ correspondence $\eta$ from the set of periodic point classes $p(\fix(\alpha\tilde{f}^n))$
to the set of Reidemeister classes $[\alpha]^n$ of $\varphi^n$.
When $m\mid n$, $\fix(f^m)\subset\fix(f^n)$. Let $x\in\fix(f^m)$ and $\tilde{x}\in p^{-1}(x)$.
Then there exist unique $\alpha,\beta\in\pi$ such that $\alpha\tilde{f}^m(\tilde{x})=\tilde{x}$
and $\beta\tilde{f}^n(\tilde{x})=\tilde{x}$. It can be easily derived that
$$
\beta=\alpha\varphi^m(\alpha)\varphi^{2m}(\alpha)\cdots\varphi^{n-m}(\alpha).
$$
This defines two natural functions, called {\bf boosting functions},
\begin{align*}
&\gamma_{m,n}:p\left(\fix(\alpha\tilde{f}^m)\right)\mapsto
p\left(\fix(\alpha\varphi^m(\alpha)\varphi^{2m}(\alpha)\cdots\varphi^{n-m}(\alpha)\tilde{f}^n)\right),\\
&\iota_{m,n}=\iota_{m,n}(\varphi):[\alpha]^m\mapsto [\alpha\varphi^m(\alpha)\varphi^{2m}(\alpha)\cdots\varphi^{n-m}(\alpha)]^n
\end{align*}
so that the following diagram is commutative
\smallskip

\centerline{ \xymatrix{ p\left(\fix(\alpha\tilde{f}^m)\right)
\ar@{|-{>}}[rr]^(0.3){\gamma_{m,n}} \ar@{|->}[d]^\eta &  &
p\left(\fix(\alpha\varphi^m(\alpha)\varphi^{2m}(\alpha)\cdots\varphi^{n-m}(\alpha)\tilde{f}^n)\right)
\ar@{|->}[d]^\eta \\  [\alpha]^m \ar@{|-{>}}[rr]^(0.3){\iota_{m,n}} & &
[\alpha\varphi^m(\alpha)\varphi^{2m}(\alpha)\cdots\varphi^{n-m}(\alpha)]^n\\
 }}
\smallskip

\noindent
Moreover, it is straightforward to check the commutativity of the diagram
\smallskip

\centerline{ \xymatrix{ [\alpha]^m
\ar@{|-{>}}[rrr]^(0.3){r_{(\beta\varphi(\beta)\cdots\varphi^{m-1}(\beta))^{-1}}}
\ar@{|->}[d]^{\iota_{m,n}(\varphi)} & & &
[\alpha(\beta\varphi(\beta)\cdots\varphi^{m-1}(\beta))^{-1}]^m
\ar@{|->}[d]^{\iota_{m,n}(\tau_\beta\varphi)} \\
[\alpha\varphi^m(\alpha)\cdots\varphi^{n-m}(\alpha)]^n \ar@{|-{>}}[rrr]^(0.4){r_{(\beta\varphi(\beta)\cdots\varphi^{n-1}(\beta))^{-1}}} & & &
[\alpha\varphi^m(\alpha)\cdots\varphi^{n-m}(\alpha)(\beta\varphi(\beta)\cdots\varphi^{m-1}(\beta))^{-1}]^n\\
 }}
\smallskip

On the other hand, for $x\in p(\fix(\alpha\tilde{f}^n))$ we choose
$\tilde{x}\in p^{-1}(x)$ so that
$\alpha\tilde{f}^n(\tilde{x})=\tilde{x}$. Then
$$
\varphi(\alpha)\tilde{f}^n\tilde{f}(\tilde{x})
=\varphi(\alpha)\tilde{f}\tilde{f}^n(\tilde{x})
=\tilde{f}\alpha\tilde{f}^n(\tilde{x})=\tilde{f}(\tilde{x})
$$
and so $f(x)\in p(\fix(\varphi(\alpha)\tilde{f}^n))$. Namely,
$p(\fix(\varphi(\alpha)\tilde{f}^n))$ is the periodic point class
determined by $f(x)$. Therefore, $f$ induces a function on the
periodic point classes of $f^n$, which we denote by $[f]$, defined
as follows:
$$
[f]: p\left(\fix(\alpha\tilde{f}^n)\right) \longmapsto
p\left(\fix(\varphi(\alpha)\tilde{f}^n)\right).
$$
Similarly, $\varphi$ induces a well-defined function on the
Reidemeister classes of $\varphi^n$, which we will denote by
$[\varphi]$, given by $[\varphi] : [\alpha]^n\mapsto
[\varphi(\alpha)]^n$. Then the following diagram commutes:
\smallskip

\centerline{ \xymatrix{ p\left(\fix(\alpha\tilde{f}^n)\right)
\ar@{|-{>}}[rr]^(0.5){[f]} \ar@{<->}[d] &  &
p\left(\fix(\varphi(\alpha)\tilde{f}^n)\right) \ar@{<->}[d]
\\  [\alpha]^n \ar@{|-{>}}[rr]^(0.5){[\varphi]} & &
[\varphi(\alpha)]^n\\
 }}
\smallskip

\noindent%
By \cite[Theorem~III.1.12]{Jiang}, $[f]$ is an index-preserving
bijection on the periodic point classes of $f^n$. We say that
$[\alpha]^n$ is {\bf essential} if the corresponding class
$p(\fix(\alpha\tilde{f}^n))$ is essential. Evidently,
\smallskip

\centerline{ \xymatrix{ \fix(\alpha\tilde{f}^n)
\ar@/_1.5pc/[rr]|{\mathrm{{\ }identity }} \ar[r]^(0.48){\tilde{f}} &
\fix(\varphi(\alpha)\tilde{f}^n)
\ar[r]^(0.52){\alpha\tilde{f}^{n-1}} & \fix(\alpha\tilde{f}^n).
 }}
\smallskip

\noindent%
This implies that for each $\alpha\in\Pi$, the restrictions of $f$
$$
f\vert : p\left(\fix(\alpha\tilde{f}^n)\right) \lra
p\left(\fix(\varphi(\alpha)\tilde{f}^n)\right)
$$
are homeomorphisms such that $[f]^n$ is the identity. In particular,
$$
p\left(\fix(\alpha\tilde{f}^n)\right)=\emptyset \quad \Longleftrightarrow\quad
p\left(\fix(\varphi(\alpha)\tilde{f}^n)\right)=\emptyset.
$$
Moreover, $[\varphi]^n$ is the identity, $\iota_{m,n}\circ[\varphi]=[\varphi]\circ\iota_{m,n}$
and $\gamma_{m,n}\circ[f]=[f]\circ\gamma_{m,n}$.

The {\bf length} of the element $[\alpha]^n\in\mathcal{R}[\varphi^n]$, denoted by $\ell([\alpha]^n)$,
is the smallest positive integer $\ell$ such that $[\varphi]^\ell([\alpha]^n)=[\alpha]^n$.
The $[\varphi]$-{\bf orbit} of $[\alpha]^n$ is the set
$$
\langle [\alpha]^n\rangle=\{[\alpha]^n,
[\varphi]([\alpha]^n),\cdots, [\varphi]^{\ell-1}([\alpha]^n)\},
$$
where $\ell=\ell([\alpha]^n)$.
We must have that $\ell\mid n$. The element $[\alpha]^n\in\mathcal{R}[\varphi^n]$ is {\bf reducible} to $m$
if there exists $[\beta]^m\in\mathcal{R}[\varphi^m]$ such that $\iota_{m,n}([\beta]^m)=[\alpha]^n$.
Note that if $[\alpha]^n$ is reducible to $m$, then $m\mid n$. If $[\alpha]^n$ is not reducible to any $m<n$,
we say that $[\alpha]^n$ is {\bf irreducible}.
The {\bf depth} of $[\alpha]^n$, denoted by $d([\alpha]^n)$, is the smallest integer $m$
to which $[\alpha]^n$ is reducible.
Since clearly $d([\alpha]^n)=d([\varphi]([\alpha]^n))$, we can define the {\bf depth}
of the orbit $\langle [\alpha]^n\rangle$: $d(\langle [\alpha]^n\rangle)=d([\alpha]^n)$.
If $n=d([\alpha]^n)$, the element $[\alpha]^n$ or the orbit $\langle [\alpha]^n\rangle$
is called {\bf irreducible}.

Clearly, as a set $p(\fix(\alpha\tilde{f}^m)) \subset
p(\fix(\alpha\varphi^m(\alpha)\varphi^{2m}(\alpha)\cdots\varphi^{n-m}(\alpha)\tilde{f}^n))$.
This implies that if $p(\fix(\alpha\tilde{f}^m))$ is the periodic point class of
$f^m$ determined by $x$, then
$$
p\left(\fix(\alpha\varphi^m(\alpha)\varphi^{2m}(\alpha)\cdots\varphi^{n-m}(\alpha)\tilde{f}^n)\right)
$$
is the periodic point class of $f^n$ determined by $x$.

Note that if $[\alpha]^n$ is irreducible, then every element of the fixed point class
$p(\fix(\alpha\tilde{f}^n))$ is a periodic point of $f$ with minimal period $n$.
Let $[\alpha]^n$ be an essential class with depth $m$
and let $\iota_{m,n}([\beta]^m)=[\alpha]^n$.
Then there is a periodic point $x$ of $f$ with minimal period $m$.
Consequently, the irreducibility of a periodic Reidemeister class of $\varphi$
is an algebraic counterpart of the minimal period of a periodic point of $f$.
We say that a periodic Reidemeister class $[\alpha]^n$ of $\varphi$ has {\bf height} $n$
if it is irreducible.
The set $\calI\calR(\varphi^n)$ of all classes in $\calR[\varphi^n]$ with height $n$
is an algebraic analogue of the set $P_n(f)$ of periodic points of $f$ with minimal period $n$.
Let $\calI(\varphi)$ be the set of all irreducible classes of $\varphi$.
That is,
$$
\calI(\varphi)=\{[\alpha]^k\in\calR[\varphi^k]\mid \alpha\in\Pi, k>0, [\alpha]^k \text{ is irreducible}\}.
$$
We define the set $\calH(\varphi)$ of all {\bf heights} of $\varphi$ to be
$$
\calH(\varphi)=\{k\in\bbn\mid \text{ some $[\alpha]^k$ has height $k$}\}.
$$
Then $\calH(\varphi)$ is an algebraic analogue of the set $\Per(f)$ of all minimal periods of $f$.
Motivated from homotopy minimal periods of $f$,
we may define the set of all {\bf homotopy heights} of $\varphi$ as follows:
$$
\calH\calI(\varphi)=\bigcap_{\beta\in\Pi}\left\{n\in\bbn\mid \calI\calR((\tau_\beta\varphi)^n)\ne\emptyset\right\}.
$$
However, as we have observed before, since the boosting functions $\iota_{m,n}$ commute with
``right multiplications", i.e.,
$$
\iota_{m,n}(\tau_\beta\varphi)\circ r_{(\beta\varphi(\beta)\cdots\varphi^{m-1}(\beta))^{-1}}
=r_{(\beta\varphi(\beta)\cdots\varphi^{n-1}(\beta))^{-1}}\circ\iota_{m,n}(\varphi),
$$
it follows that the height is a conjugacy invariant.
Consequently, we have $\calH\calI(\varphi)=\calH(\varphi)$.

\section{Poly-Bieberbach groups}\label{poly-B}

The fundamental group of an infra-solvmanifold is called a {\bf poly-Bieberbach} group,
which is a torsion free poly-crystallographic group.
It is known (see for example \cite[Theorem~2.12]{FJ}) that
every poly-Bieberbach group is a torsion-free virtually poly-$\bbz$ group.
We refer to \cite[Theorem~3]{Wilking} for a characterization of poly-crystallographic groups.
Recall also from \cite[Corollary~4]{Wilking} that
for any poly-Bieberbach group $\Pi$
there exist a connected simply connected supersolvable Lie group $S$,
a compact subgroup $K$ of $\aut(S)$ and an isomorphism $\iota$ of $\Pi$ onto
a discrete cocompact subgroup of $S\rtimes K$ such that $\iota(\Pi)\cdot S$ is dense in $S\rtimes K$.
By \cite[Lemma~2.1]{FL2}, the supersolvable Lie groups are the Lie groups of type $\R$, that is, Lie groups for which
if $\ad\, X:\frakS\to\frakS$ has only real eigenvalues for all $X$ in the Lie algebra $\frakS$ of $S$.
Assuming $\iota$ to be an inclusion or identifying $\Pi$ with $\iota(\Pi)$, we have the following commutative diagram
$$
\CD
1@>>>S@>>>S\rtimes K@>{p}>>K@>>>1\\
@.@AAA@AAA@AAA\\
1@>>>\Pi\cap S@>>>\Pi@>>>p(\Pi)@>>>1
\endCD
$$
Here, we cannot assume that the subgroup $p(\Pi)$ of $K\subset\aut(S)$ is a finite group
and that the translations $\Pi\cap S$ form a lattice in the solvable Lie group $S$ of type $\R$.

{In this paper, we will assume the following:
Let $\Pi$ be a poly-Bieberbach group which is the fundamental group of an infra-solvmanifold of type $\R$, i.e.,
$\Pi$ is a discrete cocompact subgroup of $\aff(S):=S\rtimes\aut(S)$,
where $S$ is a connected, simply connected solvable Lie group of type $\R$
and $\Pi\cap S$ is of finite index in $\Pi$ and a lattice of $S$.
The finite group $\Phi:=\Pi/\Pi\cap S$ is called the {\bf holonomy group} of the poly-Bieberbach group $\Pi$ or the infra-solvmanifold
$\Pi\bs{S}$ of type $\R$.}
Naturally $\Phi$ sits in $\aut(S)$.
Let $\varphi:\Pi\to\Pi$ be an endomorphism.
Then by \cite[Theorem~2.2]{LL-Nagoya}, $\varphi$ is semi-conjugate by an ``affine map".
Namely, there exist $d\in S$ and a Lie group endomorphism $D:S\to S$ such that
$\varphi(\alpha)(d,D)=(d,D)\alpha$ for all $\alpha\in\Pi\subset\aff(S)$.
From this identity condition, the affine map $\tilde{f}:=(d,D):S\to S$ restricts
to a map $f:\Pi\bs{S}\to\Pi\bs{S}$ for which it induces the endomorphism $\varphi$.
Conversely, if $f$ is a self-map on an infra-solvmanifold $\Pi\bs{S}$ of type $\R$,
$f$ induces an endomorphism $\varphi=\varphi_f$, see Section~\ref{introd}.
As remarked above, $f$ is homotopic to a map induced by an affine map on $S$.
Since the Lefschetz, Nielsen and Reidemeister numbers of $f$ are homotopy invariants,
we may assume that our $f$ has an affine lift $(d,D)$ on $S$.

\begin{Thm}[{\cite[Corollary~7.6]{FL}}] \label{av}
Let $\varphi:\Pi\to\Pi$ be an endomorphism on a poly-Bieberbach group $\Pi$ of $S$
with holonomy group $\Phi$. If $\varphi$ is the semi-conjugate by an affine map $(d,D)$ on $S$,
then we have
$$
R(\varphi^k)=\frac{1}{\#\Phi}\sum_{A\in\Phi}\sigma\left(\det(I-A_*D_*^k)\right)
$$
where $\sigma:\bbr\to\bbr\cup\{\infty\}$ is given by $\sigma(0)=\infty$
and $\sigma(x)=|x|$ for all $x\ne0$.
Furthermore, if $R(\varphi^k)<\infty$ then $R(\varphi^k)=N(f^k)$
where $f$ is a map on $\Pi\bs{S}$ which induces $\varphi$.
\end{Thm}

When all $R(\varphi^k)$ are finite, Theorem~\ref{av} says that
the Reidemeister theory for poly-Bieberbach groups follows directly
from the Nielsen theory for infra-solvmanifolds of type $\R$.
In this paper, whenever possible, we will state our results in the language of Reidemeister theory.

\begin{Prop}[{\cite[Proposition~9.3]{FL}}]\label{ess}
Let $f$ be a map on an infra-solvmanifold $\Pi\bs{S}$ of type $\R$ induced by an affine map.
Then every essential fixed point class of $f$ consists of a single element.
\end{Prop}

Let $\varphi:\Pi\to\Pi$ be an endomorphism on a poly-Bieberbach group $\Pi$.
We assume as before that $(d,D)$ be an affine map on $S$
and $f:\Pi\bs{S}\to\Pi\bs{S}$ be the map induced by $(d,D)$ and inducing $\varphi$.
We assume further that all $R(\varphi^n)<\infty$.
Hence by Theorem~\ref{av}, $R(\varphi^n)=N(f^n)$ for all $n>0$.
This implies that for every $n>0$ all fixed point classes of $f^n$ are essential
and hence consist of a single element by Proposition~\ref{ess}.
Consequently, we can refer to essential fixed point classes of $f^n$
as {essential periodic points of $f$ with period $n$}.
Moreover, for every $n>0$ {\bf all Reidemeister classes of $\varphi^n$ are essential}.

For $m\mid n$ and for $\beta\in\Pi$, let $\alpha=\beta\varphi^m(\beta)\cdots\varphi^{n-m}(\beta)$
and consider the commuting diagram
\smallskip

\centerline{ \xymatrix{ [\alpha]^n
\ar@{|-{>}}[rr]^(0.4){\eta} \ar@{<-|}[d]^{\iota_{m,n}} &  &
p\left(\fix(\alpha(\alpha)\tilde{f}^n)\right)=\{x\}
\ar@{<-|}[d]^{\gamma_{m,n}} \\  [\beta]^m \ar@{|-{>}}[rr]^(0.4){\eta} & &
p\left(\fix(\beta\tilde{f}^m)\right)=\{x\}\\
 }}
\smallskip

\noindent
This shows that the observation in Section~\ref{prelim} can be refined as follows:
$[\alpha]^n$ is irreducible if and only $[\alpha]^n$ has height $n$
if and only if the corresponding essential periodic point $x$ of $f$ has minimal period $n$.
Moreover, $[\alpha]^n$ has depth $d$
if and only if the corresponding essential periodic point $x$ of $f$ has minimal period $d$.
Let $\ell$ be the length of $[\alpha]^n$.
That is, $[\varphi^\ell(\alpha)]^n=[\alpha]^n$.
Equivalently, we have $f^\ell(x)=x$. This implies that $[\alpha]^n$ is reducible to $\ell$.
Further, $d=\ell$. In particular, if $[\alpha]^n$ is irreducible, then its length is the height, $\ell=n$,
and so $\#\langle[\alpha]^n\rangle=n$.

We denote by $\calO([\varphi],k)$ the set of all (essential) periodic orbits of ${[\varphi]}$ with length $\le k$.
Then we have
\begin{align*}
\calO([\varphi],k)&=\{\langle[\alpha]^m\rangle\mid \alpha\in\Pi, m\le k\}\\
&=\{\langle x\rangle\mid x \text{ is an essential periodic point of $f$ with length $\le k$}\}\\
&=\calO(f,k).
\end{align*}

Recall that the set of {essential periodic points} of $f$ with minimal period $k$ is
$$
\EP_k(f)=\fix_e(f^k)-\bigcup_{d\mid k, d<k}\fix_e(f^d).
$$
Then we have the algebraic counterpart. Namely,
\begin{align*}
\EP_k(\varphi)&=\{[\alpha]^k\in\calR[\varphi^k]\mid [\alpha]^k \text{ is irreducible (and essential)}\}\\
&=\{[\alpha]^k\in\calR[\varphi^k]\mid [\alpha]^k \text{ has height $k$}\}\\
&=\calI\calR(\varphi^k),
\end{align*}
and $\#\calI\calR(\varphi^k)=\#\EP_k(f)$.
Hence the set of (essential) Reidemeister classes of $\varphi^k$ can be identified
with a disjoint union of irreducible classes, that is, $\calR[\varphi^k]$ is decomposed by heights:
$$
\calR[\varphi^k]=\bigsqcup_{d\mid k}\ \calI\calR(\varphi^d),
$$
and hence its cardinality satisfies
$$
R(\varphi^k)=\#\calR[\varphi^k]=\sum_{d\mid k}\# \calI\calR(\varphi^d).
$$

Recall that if we denote by $O_k(\varphi)$ the number of essential
and irreducible periodic orbits of $\calR[\varphi^k]$,
i.e., if $O_k(\varphi)=\#\{\langle[\alpha]^k\rangle\mid [\alpha]^k\in \EP_k(\varphi)\}$
then by definition, the {\bf prime Nielsen--Jiang periodic number} of period $k$ is
$$
\NP_k(\varphi)=k\x O_k(\varphi).
$$
As observed earlier, each such orbit $\langle[\alpha]^k\rangle$ has length $k$.
Therefore, $\NP_k(\varphi)=\#\calI\calR(\varphi^k)$ and topologically $\NP_k(f)=\#\EP_k(f)$,
the number of essential periodic points of $f$ with minimal period $k$.

\begin{Thm}\label{NP}
Let $\varphi:\Pi\to\Pi$ be an endomorphism on a poly-Bieberbach group such that all $R(\varphi^k)$ are finite.
Then
$$
\NP_k(\varphi)=\# \calI\calR(\varphi^k).
$$
Let $f:\Pi\bs{S}\to\Pi\bs{S}$ be a map on an infra-solvmanifold $\Pi\bs{S}$ of type $\R$. Then
$$
\NP_k(f)=\# \EP_k(f).
$$
\end{Thm}

Consider all periodic orbits $\langle[\alpha]^m\rangle$ of $\varphi$ with $m\mid k$.
A set $\frakS$ of periodic orbits $\langle[\alpha]^m\rangle$ of $\varphi$ with $m\mid k$
is said to be a set of $k$-representatives
if every essential orbit $\langle[\beta]^m\rangle$ with $m\mid k$ is reducible to some element of $\frakS$.
Then the {\bf full Nielsen--Jiang periodic number} of period $k$ is defined to be
$$
\NF_k(\varphi)=\min\left\{\sum_{\calA\in\frakS}d(\calA)
\mid \frakS \text{ is a set of $k$-representatives}\right\}.
$$
Recall that all periodic orbits are essential,
and every periodic orbit $\langle[\beta]^m\rangle$ is boosted to a periodic orbit $\langle[\alpha]^k\rangle$.
Hence to compute $\NF_k(\varphi)$, we need first to consider only $\calR[\varphi^k]$
and a set of representatives $[\alpha_i]^k$ of the orbits in $\calR[\varphi^k]$.
Then $\NF_k(\varphi)$ is the sum of depths of all $[\alpha_i]^n$.
Topologically, $\NF_k(f)$ is the sum of minimal periods of all essential periodic point classes
$\langle x\rangle$ in $\fix_e(f^k)$.

\section{The Nielsen theory of iterations on an infra-solvmanifold of type $\R$}\label{N_itetate}

In this section, we shall assume that $f:M\to M$ is a continuous map on an infra-solvmanifold $M=\Pi\bs{S}$
of type $\R$ with holonomy group $\Phi$. Then $f$ admits an affine homotopy lift $(d,D):S\to S$.
Concerning the Nielsen numbers $N(f^k)$ of all iterates of $f$, we begin with the following facts:
\medskip

\noindent
{\bf Averaging Formula:} (\cite[Theorem~4.2]{LL-Nagoya})
\begin{align*}
N(f^k)=\frac{1}{\#\Phi}\sum_{A\in\Phi}|\det(I-A_*D_*^k)|.
\end{align*}
\smallskip

\noindent
{\bf Gauss Congruences:} (\cite[Theorem~11.4]{FL})
\begin{align*}
\sum_{d\mid k}\mu\!\left(\frac{k}{d}\right)N(f^d)\equiv0\mod{k}.
\end{align*}
{Indeed, we have shown in \cite[Theorem~11.4]{FL} that the left-hand side is non-negative because it is equal to the number of isolated periodic points of $f$ with least period $k$.
By \cite[Lemma~2.1]{PW}, the sequence $\{N(f^k)\}$ is exactly realizable.}
\medskip

\noindent
{\bf Rationality of Nielsen zeta function:} (\cite[Theorem~4.5]{DeDu},\cite{FL})
\begin{align*}
&N_f(z)=\exp\left(\sum_{k=1}^\infty\frac{N(f^k)}{k}z^k\right)
\end{align*}
is a rational function with coefficients in $\bbq$.

Using these facts as our main tools, we study the asymptotic behavior of the sequence $\{N(f^k)\}$,
the essential periodic orbits of $f$ and the homotopy minimal periods of $f$
by using the Nielsen theory of maps $f$ on infra-solvmanifolds of type $\R$.
We will give a brief description of the main results in this section
whose details and proofs can be found in \cite{FL2}.

Consider the sequences of algebraic multiplicities $\{A_k(f)\}$
and Dold multiplicities $\{I_k(f)\}$ associated to the sequence $\{N(f^k)\}$:
$$
A_k(f)=\frac{1}{k}\sum_{d\mid k}\mu\!\left(\frac{k}{d}\right)N(f^d),\quad
I_k(f)=\sum_{d\mid k}\mu\!\left(\frac{k}{d}\right)N(f^d).
$$
Then $I_k(f)=kA_k(f)$ and all $A_k(f)$ are integers by \eqref{DN}.
From the M\"{o}bius inversion formula, we immediately have
\begin{align}\label{M}
N(f^k)=\sum_{d\mid k}\ d\!\!~A_d(f).\tag{M}
\end{align}

On the other hand, since $N_f(0)=1$ by definition, $z=0$ is not a zero nor a pole
of the rational function $N_f(z)$. Thus we can write
$$
N_f(z)=\frac{u(z)}{v(z)}=\frac{\prod_i(1-\beta_iz)}{\prod_j(1-\gamma_jz)}
=\prod_{i=1}^r(1-\lambda_iz)^{-\rho_i}
$$
with all $\lambda_i$ distinct nonzero algebraic integers {(see for example \cite{BL} or \cite[Theorem~2.1]{BaBo})}
and $\rho_i$ nonzero integers.
This implies that
\begin{align*}
N(f^k)=\sum_{i=1}^{r(f)}\rho_i\lambda_i^k.
\end{align*}
Note that $r(f)$ is {the number of zeros and poles of $N_f(z)$}.
Since $N_f(z)$ is a homotopy invariant, so is $r(f)$.

We define
$$
\lambda(f)=\max\{|\lambda_i|\mid i=1,\cdots,r(f)\}.
$$
The number $\lambda(f)$ will play a similar role as the
``essential spectral radius" in \cite{JM} or the ``reduced spectral radius" in \cite{BaBo}.

The following theorem shows that {$1/\lambda(f)$} is the ``radius" of the Nielsen zeta function $N_f(z)$.
Note also that $\lambda(f)$ is a homotopy invariant.

\begin{Thm}[{\cite[Theorem~3.2]{FL2}}]
\label{Radius}
Let $f$ be a map on an infra-solvmanifold of type $\R$ with an affine homotopy lift $(d,D)$.
Let {\rm $\textsf{R}$} denote the radius of convergence of the Nielsen zeta function $N_f(z)$ of $f$.
Then $\lambda(f)=0$ or $\lambda(f)\ge1$, and
\begin{align*}
\frac{1}{\text{\rm{$\textsf{R}$}}}=\lambda(f).
\end{align*}
In particular, {\rm $\textsf{R}>0$}. If $D_*$ has no eigenvalue $1$, then
$$
\frac{1}{\text{\rm{$\textsf{R}$}}}=\sp\left(\bigwedge D_*\right)=\lambda(f).
$$
\end{Thm}

Next, we recall the asymptotic behavior of the Nielsen numbers of iterates of maps.

\begin{Thm}[{\cite[Theorem~4.1]{FL2}}]
\label{N-BaBo2.6}
For a map $f$ of an infra-solvmanifold of type $\R$,
one of the following {three} possibilities holds:
\begin{enumerate}
\item[$(1)$] $\lambda(f)=0$, which occurs if and only if $N_f(z)\equiv1$.
\item[$(2)$] The sequence $\{N(f^k)/\lambda(f)^k\}$ has the same limit points
as a periodic sequence $\{\sum_j\alpha_j\epsilon_j^k\}$
where $\alpha_j\in\bbz,\epsilon_j\in\bbc$ and $\epsilon_j^q=1$ for some $q>0$.
\item[$(3)$] {The set of limit points of the sequence $\{N(f^k)/\lambda(f)^k\}$ contains an interval.}
\end{enumerate}
\end{Thm}

In order to give an estimate from below for the number of {\bf essential periodic orbits}
of maps on infra-solvmanifolds of type $\R$, we recall the following:
\begin{Thm}[{\cite{SS}}]\label{SS}
If $f:M\to M$ is a $C^1$-map on a smooth compact manifold $M$ and $\{L(f^k)\}$ is unbounded,
then the set of periodic points of $f$, $\bigcup_k\fix(f^k)$, is infinite.
\end{Thm}

This theorem is not true for continuous maps.
Consider the one-point compactification of the map of the complex plane $f(z)=2z^2/\bb z\bb$.
This is a continuous degree two map of $S^2$ with only two periodic points but with $L(f^k)=2^{k+1}$.

However, when $M$ is an infra-solvmanifold of type $\R$,
the theorem is true for all {continuous} maps $f$ on $M$.
In fact, using the averaging formula, we obtain
\begin{align*}
|L(f^k)|\le\frac{1}{|\Phi|}\sum_{A\in\Phi}|\det(I-A_*D_*^k)|=N(f^k).
\end{align*}
If $L(f^k)$ is unbounded, then so is $N(f^k)$ and hence the number
of essential fixed point classes of all $f^k$ is infinite.

Recall that any map $f$ on an infra-solvmanifold of type $\R$ is homotopic to a map $\bar{f}$
induced by an affine map $(d,D)$. By \cite[Proposition~9.3]{FL},
every essential fixed point class of $\bar{f}$ consists of a single element $x$ with index {$\sgn\det(I-df_x)$}.
Hence $N(f)=N(\bar{f})$ is the number of essential fixed point classes of $\bar{f}$.
It is a classical fact that a homotopy between $f$ and $\bar{f}$ induces a one-one correspondence
between the fixed point classes of $f$ and those of $\bar{f}$, which is index preserving.
Consequently, we obtain
$$
|L(f^k)|\le N(f^k)\le \#\fix(f^k).
$$
This suggests the following conjectural inequality (see \cite{Shub,SS})
for infra-solvmanifolds of type $\R$:
$$
\limsup_{k\to\infty}\frac{1}{k}\log|L(f^k)|\le \limsup_{k\to\infty}\frac{1}{k}\log\#\fix(f^k).
$$

We denote by $\calO(f,k)$ the set of all essential periodic orbits of $f$ with length $\le k$.
Thus
$\calO(f,k)=\{\langle\bbf\rangle\mid \bbf \text{ is a essential fixed point class of $f^m$ with $m\le k$}\}$.
We can strengthen Theorem~\ref{SS} as follows:

\begin{Thm}
\label{N_BaBo4.2}
Let $f$ be a map on an infra-solvmanifold of type $\R$.
Suppose that the sequence $N(f^k)$ is unbounded. Then there exists a natural number $N_0$ such that
$$
k\ge N_0\Longrightarrow \#\calO(f,k)\ge \frac{k-N_0}{r(f)}.
$$
\end{Thm}

\begin{proof}
As mentioned earlier, we may assume that every essential fixed point class $\bbf$ of any $f^k$ consists of a single element $\bbf=\{x\}$. Denote by $\fix_e(f^k)$ the set of essential fixed point (class) of $f^k$. Thus $N(f^k)=\#\fix_e(f^k)$. Recalling also that $f$ acts on the set $\fix_e(f^k)$ from the proof of \cite[Theorem~11.4]{FL}, we have
$$
\calO(f,k)=\{\langle x\rangle\mid x \text{ is a essential periodic point of $f$ with length $\le k$}\}.
$$

Observe further that if $x$ is an essential periodic point of $f$ with least period $p$, then $x\in\fix_e(f^{q})$ if and only if $p\mid q$. The length of the orbit $\langle x\rangle$ of $x$ is $p$, and
\begin{align*}
&\fix_e(f^k)=\bigcup_{d\mid k}\ \fix_e(f^d),\\
&\fix_e(f^d)\bigcap\fix_e(f^{d'})=\fix_e(f^{\gcd(d,d')}).
\end{align*}
Recalling that
\begin{align*}
A_m(f)&=\frac{1}{m}\sum_{k\mid m}\mu\!\left(\frac{m}{k}\right)N(f^k)
=\frac{1}{m}\sum_{k\mid m}\mu\!\left(\frac{m}{k}\right)\#\fix_e(f^k),
\end{align*}
we define $A_m(f,\langle x\rangle)$ for any $x\in\bigcup_i\fix_e(f^i)$ to be
\begin{align*}
A_m(f,\langle x\rangle)&=\frac{1}{m}\sum_{k\mid m}\mu\!\left(\frac{m}{k}\right)\#\!\left(\langle x\rangle\cap\fix_e(f^k)\right).
\end{align*}
Then we have
\begin{align*}
A_m(f)=\sum_{\substack{\langle x\rangle\\ x\in\fix_e(f^m)}} A_m(f,\langle x \rangle).
\end{align*}

We begin with new notation. For a given integer $k>0$ and $x\in\bigcup_m\fix_e(f^m)$, let
\begin{align*}
&\calA(f,k)=\left\{m\le k\mid A_m(f)\ne0\right\},\\
&\calA(f,\langle x\rangle)=\left\{m\mid A_m(f,\langle x\rangle)\ne0\right\}.
\end{align*}
Notice that if $A_m(f)\ne0$ then there exists an essential periodic point $x$ of $f$ with period $m$ such that $A_m(f,\langle x\rangle)\ne0$. Consequently, we have
$$
\calA(f,k)\subset \bigcup_{\langle x\rangle\in\calO(f,k)} \calA(f,\langle x\rangle)
$$

Since $N(f^k)$ is unbounded, we have that $\lambda(f)>1$ by the definition of $\lambda(f)$. {By \cite[Corollary~4.6]{FL2}, there is $N_0$ such that if $n\ge N_0$ then there is $i$ with $n\le i\le n+n(f)-1$} such that $A_i(f)\ne0$. This leads to the estimate
$$
\#\calA(f,k)\ge\frac{k-N_0}{n(f)}\quad \forall k\ge N_0.
$$

Assume that $x$ has least period $p$. Then we have
\begin{align*}
A_m(f,\langle x\rangle)
&=\frac{1}{m}\sum_{p\mid n\mid m}\mu\!\left(\frac{m}{n}\right)\#\langle x\rangle
=\frac{p}{m}\sum_{p\mid n\mid m}\mu\!\left(\frac{m}{n}\right).
\end{align*}
Thus if $m$ is not a multiple of $p$ then by definition $A_m(f,\langle x\rangle)=0$. It is clear that $A_p(f,\langle x\rangle)=\mu(1)=1$, i.e., $p\in\calA(f,\langle x\rangle)$. Because $p\mid n\mid rp\Leftrightarrow n=r'p$ with $r'\mid r$, we have $A_{rp}(f,\langle x\rangle)=1/r \sum_{p\mid n\mid rp}\mu(rp/n)=1/r\sum_{r'\mid r}\mu(r/r')$ which is $0$ when and only when $r>1$. Consequently, $\calA(f,\langle x\rangle)=\{p\}$.

In conclusion, we obtain the required inequality
\begin{align*}
\frac{k-N_0}{r(f)}&\le \#\calA(f,k)\le\#\calO(f,k).\qedhere
\end{align*}
\end{proof}

Finally, we study (homotopy) minimal periods of maps $f$ on infra-solvmanifolds of type $\R$.
We seek to determine $\HPer(f)$ only from the knowledge of the sequence $\{N(f^k)\}$.
This approach was used in \cite{ABLSS, Hal, JL} for maps on tori,
in \cite{JM1,JM2,JKM,JM,LZ-china,LZ-agt} for maps on nilmanifolds
and some solvmanifolds, and in \cite{LL-JGP,LZ-jmsj} for expanding maps on infra-nilmanifolds.

Utilizing new results obtained from the Gauss congruences and the rationality of the Nielsen zeta function,
together with Dirichlet's prime number theorem, we obtain:
\begin{Thm}
\label{N_3.2.48}
Let $f$ be a map on an infra-solvmanifold of type $\R$.
Suppose that 
{the sequence $\{N(f^k)/\lambda(f)^k\}$ is asymptotically periodic $($i.e., Case $(2)$ of Theorem~\ref{N-BaBo2.6}$)$}.
Then there exist $m$ and an infinite sequence $\{p_i\}$ of primes such that $\{mp_i\}\subset\Per(f)$.
Furthermore, $\{mp_i\}\subset\HPer(f)$.
\end{Thm}

Next we recall that:
\begin{Thm}[{\cite[Theorem~6.1]{Je}}]\label{Je}
Let $f:M\to M$ be a self-map on a compact PL-manifold of dimension $\ge3$. Then $f$ is homotopic to a map $g$ with $P_n(g)=\emptyset$ if and only if $\NP_n(f)=0$.
\end{Thm}

The infra-solvmanifolds of dimension $1$ or $2$ are the circle, the torus and the Klein bottle. Theorem~\ref{Je} for dimensions $1$ and $2$ is verified respectively in \cite{BGMY}, \cite{ABLSS} and \cite{Llibre, JKM, KKZ}. Hence we have
\begin{align*}
n\in\HPer(f) & \Longleftrightarrow \exists\, g\simeq f \text{ such that } P_n(f)\ne\emptyset
               \quad(\text{Definition})\\
& \Longleftrightarrow \NP_n(f)\ne0 \quad(\text{Theorem~\ref{Je}})\\
& \Longleftrightarrow \EP_n(f)\ne\emptyset \quad(\text{Theorem~\ref{NP}})\\
& \Longleftrightarrow I_n(f)\ne0 \quad(\text{\cite[Proposition~5.4]{FL2}})\\
& \Longleftrightarrow A_n(f)\ne0
\end{align*}

With the identity \eqref{M}, we have the following result.
\begin{Thm}\label{Alg}
Let $f$ be a map on an infra-solvmanifold of type $\R$.
Then
\begin{align*}
\HPer(f)&=\{k\mid A_k(f)\ne0\}\subset\{k\mid N(f^k)\ne0\}.
\end{align*}
Moreover, if $N(f^k)\ne0$, then there exists a divisor $d$ of $k$ such that $d\in\HPer(f)$.
\end{Thm}

\begin{Cor}\label{N-cofinite}
Let $f$ be a map on an infra-solvmanifold of type $\R$.
Suppose that the sequence $\{N(f^k)\}$ is strictly monotone increasing. Then $\HPer(f)$ is cofinite.
\end{Cor}

\begin{proof}
{By the assumption, we have $\lambda(f)>1$.} Thus by \cite[Theorem~4.4]{FL2} (cf. Theorem~\ref{BaBo2.7}), there exist $\gamma>0$ and $N$ such that if $k>N$ then there exists $\ell=\ell(k)<r(f)$ such that $N(f^{k-\ell})/\lambda(f)^{k-\ell}>\gamma$. Then for all $k>N$, the monotonicity implies that
\begin{align*}
\frac{N(f^k)}{\lambda(f)^k}\ge \frac{N(f^{k-\ell})}{\lambda(f)^k}=\frac{N(f^{k-\ell})}{\lambda(f)^{k-\ell}\lambda(f)^\ell}
\ge\frac{\gamma}{\lambda(f)^\ell}\ge\frac{\gamma}{\lambda(f)^{r(f)}}.
\end{align*}
Applying \cite[Proposition~4.5]{FL2} (cf. Proposition~\ref{BaBo2.8}) with $\epsilon=\gamma/\lambda(f)^{r(f)}$, we see that $I_k(f)\ne0$ and so $A_k(f)\ne0$ for all $k$ sufficiently large. Now our assertion follows from Theorem~\ref{Alg}.
\end{proof}

\begin{Rmk}
Note that in the above Corollary we may use the weaker assumption that the sequence $\{N(f^k)\}$ is eventually strictly monotone increasing, i.e., there exists $k_0>0$ such that $N(f^{k+1})>N(f^k)$ for all $k\ge k_0$.
\end{Rmk}

Thus the main result of \cite{LZ-jmsj} follows from Corollary~\ref{N-cofinite}.
\begin{Cor}[{\cite[Theorem~4.6]{LL-JGP}, \cite[Theorem~3.2]{LZ-jmsj}}]
Let $f$ be an expanding map on an infra-nilmanifold. Then $\HPer(f)$ is cofinite.
\end{Cor}

{In \cite{DeDu-hmp}, the authors also discussed homotopy minimal periods for
hyperbolic maps on infra-nilmanifolds.
A map $f$ on an infra-nilmanifold with affine homotopy lift $(d,D)$ is {\bf hyperbolic}
if $D_*$ has no eigenvalues of modulus $1$.
We now give another proof of each of the main results, Theorems~3.9 and 3.16, in \cite{DeDu-hmp}.
In our proof, we use some useful results such as Lemma~3.7 and Proposition~3.14 in \cite{DeDu-hmp}.}

\begin{Thm}[{\cite[Theorem~3.9]{DeDu-hmp}}]
If $f$ is a hyperbolic map on an infra-nilmanifold with affine homotopy lift $(d,D)$ such that $D_*$ is not nilpotent, then $\HPer(f)$ is cofinite.
\end{Thm}

\begin{proof}
\cite[Lemma~3.7]{DeDu-hmp} says that if $f$ is such a map, then
the sequence $\{N(f^k)\}$ is eventually strictly monotone increasing;
by our Corollary~\ref{N-cofinite}, $\HPer(f)$ is cofinite.
\end{proof}

\begin{Thm}[{\cite[Theorem~3.16]{DeDu-hmp}}]
If $f$ is a hyperbolic map on an infra-nilmanifold with affine homotopy lift $(d,D)$ such that $D_*$ is nilpotent, then $\HPer(f)=\{1\}$.
\end{Thm}

\begin{proof}
If $f$ is such a map, then by \cite[Proposition~3.14]{DeDu-hmp} all $N(f^k)=1$.
Now because of the identity \eqref{M}, all $A_d(f)=0$ except $A_1(f)=N(f^k)=1$.
Hence Theorem~\ref{Alg} implies that $\HPer(f)=\{1\}$.
\end{proof}

\section{Reidemeister numbers $R(\varphi^k)$}

Concerning the Reidemeister numbers $R(\varphi^k)$ of all iterates of $\varphi$,
{we shall assume that all $R(\varphi^k)$ are finite.}
Whenever all $R(\varphi^n)$ are finite, we can consider the Reidemeister zeta function of $\varphi$
$$
R_\varphi(z)=\exp\left(\sum_{k=1}^\infty\frac{R(\varphi^k)}{k}z^k\right).
$$

Let $\varphi:\Pi\to\Pi$ be an endomorphism on a poly-Bieberbach group $\Pi$
with $\Pi\subset \aff(S)=S\rtimes\aut(S)$,
where $S$ is a connected, simply connected solvable Lie group of type $\R$.
By Section~\ref{poly-B}, $\varphi$ is a homomorphism induced by a self-map $f$ on the infra-solvmanifold $\Pi\bs{S}$ of type $\R$.
First we recall the following result.

\begin{Thm}[{\cite[Theorem~11.4]{FL}}]\label{Dold}
Let $\varphi:\Pi\to\Pi$ be an endomorphism on a poly-Bieberbach group $\Pi$
such that all $R(\varphi^k)$ are finite.
Then the sequences $\{R(\varphi^k)\}$ and $\{N(f^k)\}$ are exactly realizable and
\begin{align}\label{DN}
\sum_{d\mid k}\mu\!\left(\frac{k}{d}\right)R(\varphi^d)\equiv
\sum_{d\mid k}\mu\!\left(\frac{k}{d}\right)N(f^d)\equiv0\mod{k}\tag{DN}
\end{align}
for all $k>0$.
\end{Thm}

Consider the sequences of algebraic multiplicities $\{A_k(f)\}$
and Dold multiplicities $\{I_k(f)\}$ associated to the sequence $\{N(f^k)\}$:
$$
A_k(f)=\frac{1}{k}\sum_{d\mid k}\mu\!\left(\frac{k}{d}\right)N(f^d),\quad
I_k(f)=\sum_{d\mid k}\mu\!\left(\frac{k}{d}\right)N(f^d).
$$
Then $I_k(f)=kA_k(f)$ and all $A_k(f)$ are integers by \eqref{DN}.
From the M\"{o}bius inversion formula, we immediately have
\begin{align*}
N(f^k)=\sum_{d\mid k}\ d\!\!~A_d(f).
\end{align*}
Because we are assuming that all $R(\varphi^k)$ are finite, by Theorem~\ref{av},
$R(\varphi^k)=N(f^k)$. Consequently, we obtain the sequences of algebraic multiplicities $\{A_k(\varphi)\}$
and Dold multiplicities $\{I_k(\varphi)\}$ associated to the sequence $\{R(\varphi^k)\}$.
Thus $I_k(\varphi)=kA_k(\varphi)$ and all $A_k(\varphi)$ are integers.
Furthermore, we immediately have $R(\varphi^k)=\sum_{d\mid k}\ d\!\!~A_d(\varphi)$.

\begin{Thm}[{\cite[Theorem~7.8]{FL}}]\label{DD}
Let $\varphi:\Pi\to\Pi$ be an endomorphism on a poly-Bieberbach group $\Pi$
such that all $R(\varphi^k)$ are finite. Then the Reidemeister zeta function of $\varphi$
$$
R_\varphi(z)=\exp\left(\sum_{k=1}^\infty\frac{R(\varphi^k)}{k}z^k\right)
$$
is a rational function.
\end{Thm}

Since $R_\varphi(0)=1$ by definition, $z=0$ is not a zero nor a pole
of the rational function $R_\varphi(z)$.
Thus we can write
$$
R_\varphi(z)=\frac{u(z)}{v(z)}=\frac{\prod_i(1-\beta_iz)}{\prod_j(1-\gamma_jz)}
=\prod_{i=1}^r(1-\lambda_iz)^{-\rho_i}
$$
with all $\lambda_i$ distinct nonzero algebraic integers {(see for example \cite{BL} or \cite[Theorem~2.1]{BaBo})} and $\rho_i$ nonzero integers.
This implies that
\begin{align}\label{N_k}
R(\varphi^k)=\sum_{i=1}^{r(\varphi)}\rho_i\lambda_i^k.\tag{R1}
\end{align}
Note that $r(\varphi)$ is {the number of zeros and poles of $R_\varphi(z)$}.
Since $R_\varphi(z)$ is a homotopy invariant, so is $r(\varphi)$.

Consider another generating function associated to the sequence $\{R(\varphi^k)\}$:
$$
S_\varphi(z)=\sum_{k=1}^\infty R(\varphi^k)z^{k-1}.
$$
Then it is easy to see that
$$
S_\varphi(z)=\frac{d}{dz}\log R_\varphi(z).
$$
Moreover,
$$
S_\varphi(z)=\sum_{k=1}^\infty \sum_{i=1}^{r(\varphi)}\rho_i\lambda_i^kz^{k-1}
=\sum_{i=1}^{r(\varphi)}\frac{\rho_i\lambda_i}{1-\lambda_iz}
$$
is a rational function with simple poles and integral residues,
and $0$ at infinity. The rational function $S_\varphi(z)$ can be written
as $S_\varphi(z)=u(z)/v(z)$ where the polynomials $u(z)$ and $v(z)$ are of the form
$$
u(z)=R(\varphi)+\sum_{i=1}^s a_iz^i,\quad v(z)=1+\sum_{j=1}^tb_jz^j
$$
with $a_i$ and $b_j$ integers, see $(3)\Rightarrow(5)$, Theorem~2.1 in \cite{BaBo}
or \cite[Lemma~3.1.31]{JM}. Let $\tilde{v}(z)$ be the conjugate polynomial of $v(z)$,
i.e., $\tilde{v}(z)=z^{t}v(1/z)$. Then the numbers $\{\lambda_i\}$ are
the roots of $\tilde{v}(z)$, and $r(\varphi)=t$.

The following can be found in the proof of $(3)\Rightarrow(5)$, Theorem~2.1 in \cite{BaBo},
see also \cite[Lemma~2.4]{FL2}.
\begin{Lemma}\label{conjugate}
If $\lambda_i$ and $\lambda_j$ are roots of the {rational} polynomial $\tilde{v}(z)$
which are algebraically conjugate $($i.e., $\lambda_i$ and $\lambda_j$ are roots of the same irreducible polynomial$)$,
then $\rho_i=\rho_j$.
\end{Lemma}

Let $\tilde{v}(z)=\prod_{\alpha=1}^s\tilde{v}_\alpha(z)$ be the decomposition
of the monic integral polynomial $\tilde{v}(z)$ into irreducible polynomials $\tilde{v}_\alpha(z)$
of degree $r_\alpha$. Of course, $r=r(\varphi)=\sum_{\alpha=1}^sr_\alpha$ and
\begin{align*}
\tilde{v}(z)&=z^r+b_1z^{r-1}+b_2z^{r-2}+\cdots+b_{r-1}z+b_r\\
&=\prod_{\alpha=1}^s(z^{r_\alpha}+b_1^\alpha z^{r_\alpha-1}
+b_2^\alpha z^{r_\alpha-2}+\cdots+b_{r_\alpha-1}^\alpha z+b_{r_\alpha}^\alpha)
=\prod_{\alpha=1}^s\tilde{v}_\alpha(z).
\end{align*}
If $\{\lambda_i^{(\alpha)}\}$ are the roots of $\tilde{v}_\alpha(z)$,
then the associated $\rho$'s are the same $\rho_\alpha$.
Consequently, we can rewrite \eqref{N_k} as
\begin{align*}
R(\varphi^k)&=\sum_{\alpha=1}^s\rho_\alpha\left(\sum_{i=1}^{r_\alpha}(\lambda_i^{(\alpha)})^k\right)\\
&=\sum_{\rho_\alpha>0}\rho_\alpha^+\left(\sum_{i=1}^{r_\alpha}(\lambda_i^{(\alpha)})^k\right)
-\sum_{\rho_\alpha<0}\rho_\alpha^-\left(\sum_{i=1}^{r_\alpha}(\lambda_i^{(\alpha)})^k\right).
\end{align*}
Consider the $r_\alpha\x r_\alpha$-{integral} square matrices
$$
M_\alpha=\left[\begin{matrix}
0&0&\cdots&0&-b^\alpha_{r_\alpha}\\
1&0&\cdots&0&\hspace{10pt}-b^\alpha_{r_\alpha-1}\\
\vdots&\vdots&&\vdots&\vdots\\
0&0&\cdots&0&-b^\alpha_2\\
0&0&\cdots&1&-b^\alpha_1
\end{matrix}\right].
$$
The characteristic polynomial is $\det(zI-M_\alpha)=\tilde{v}_\alpha(z)$
and therefore $\{\lambda_i^{(\alpha)}\}$ is the set of eigenvalues of $M_\alpha$.
This implies that $R(\varphi^k)=\sum_{\alpha=1}^s\rho_\alpha\ \tr M_\alpha^k$. Set
$$
M_+=\bigoplus_{\rho_\alpha>0} \rho_\alpha^+ M_\alpha,\qquad
M_-=\bigoplus_{\rho_\alpha<0} \rho_\alpha^- M_\alpha.
$$
Then
\begin{align}\label{N2}
R(\varphi^k)=\tr M_+^k-\tr M_-^k=\tr(M_+\bigoplus -M_-)^k.\tag{R2}
\end{align}

We will show in Proposition~\ref{epp} that if $A_k(f)\ne0$ then $N(f^k)\ne0$
and hence $f$ has an essential periodic point of period $k$.
In the following we investigate some other necessary conditions under which $N(f^k)\ne0$.
Recall that
\begin{align*}
N(f^k)&=\text{ the number of essential fixed point classes of $f^k$.}
\end{align*}
If $\bbf$ is a fixed point class of $f^k$, then $f^k(\bbf)=\bbf$ and the {\bf length} of $\bbf$
is the smallest number $p$ for which $f^p(\bbf)=\bbf$, written $p(\bbf)$.
We denote by $\langle\bbf\rangle$ the $f$-orbit of $\bbf$,
i.e., $\langle\bbf\rangle=\{\bbf,f(\bbf),\cdots, f^{p-1}(\bbf)\}$
where $p=p(\bbf)$. If $\bbf$ is essential, so is every $f^i(\bbf)$
and $\langle\bbf\rangle$ is an {\bf essential} periodic orbit of $f$
with length $p(\bbf)$ and $p(\bbf)\mid k$.
These are variations of Corollaries~2.3, 2.4 and 2.5 of \cite{BaBo}.

Assuming that all $R(\varphi^k)$ are finite, we have
\begin{Cor}
If $r(\varphi)\ne0$, then $R(\varphi^i)\ne0$ for some $1\le i\le r(\varphi)$.
In particular, $\varphi$ has an essential periodic orbit with the length $p\mid i, i\leq r(\varphi)$.
\end{Cor}

Recalling the identity $R(\varphi^k)=\sum_{i=1}^{r(\varphi)}\rho_i\lambda_i^k$, we define
\begin{align*}
&\rho(\varphi)=\sum_{i=1}^{r(\varphi)}\rho_i,\quad
M(\varphi)=\max\left\{\sum_{\rho_i\ge0}\rho_i,-\sum_{\rho_j<0}\rho_j\right\}.
\end{align*}

\begin{Cor}
If $\rho(\varphi)=0$ and $r(\varphi)\ge1$, then $r(\varphi)\ge2$
and $R(\varphi^i)\ne0$ for some $1\le i< r(\varphi)$.
In particular, $\varphi$ has an essential periodic orbit with the length $p\mid i, i\leq r(\varphi)-1$.
\end{Cor}

\begin{Cor}\label{2.9}
If $r(\varphi)>0$, then $R(\varphi^i)\ne0$ for some $1\le i\le M(\varphi)$.
In particular, $\varphi$ has an essential periodic orbit with the length $p\mid i, i\leq M(\varphi)$.
\end{Cor}

\section{Radius of convergence of $R_\varphi(z)$}

From the Cauchy--Hadamard formula, we can see that the radii $\textsf{R}$
of convergence of the infinite series $R_\varphi(z)$ and $S_\varphi(z)$ are the same and given by
$$
\frac{1}{\textsf{R}}=\limsup_{k\to\infty}\left(\frac{R(\varphi^k)}{k}\right)^{1/k}
=\limsup_{k\to\infty}R(\varphi^k)^{1/k}.
$$

We will understand the radius $\textsf{R}$ of convergence
from the identity $R(\varphi^k)=\sum_{i=1}^{r(\varphi)}\rho_i\lambda_i^k$.
Recall that the $\lambda_i^{-1}$ are the poles or the zeros of the rational function $R_\varphi(z)$.
We define
$$
\lambda(\varphi)=\max\{|\lambda_i|\mid i=1,\cdots,r(\varphi)\}.
$$

If $r(\varphi)=0$, i.e., if $R(\varphi^k)=0$ for all $k>0$,
then $R_\varphi(z)\equiv1$ and $1/\textsf{R}=0$.
In this case, we define customarily $\lambda(\varphi)=0$.
We shall assume now that $r(\varphi)\ne0$.
In what follows, when $\lambda(\varphi)>0$, we consider
$$
n(\varphi) = \#\{i\mid |\lambda_i|=\lambda(\varphi)\}.
$$

Remark that if $\lambda(\varphi)<1$ then $R(\varphi^k)=\sum_{i=1}^{r(\varphi)} \rho_i\lambda_i^k\to 0$ and
so the sequence of integers are eventually zero, i.e., $R(\varphi^k)=0$ for all $k$ sufficiently large.
This shows that $1/\textsf{R}=0$ and furthermore, $R_\varphi(z)$ is the exponential of a polynomial.
Hence the rational function $R_\varphi(z)$ has no poles and zeros.
This forces $R_\varphi(z)\equiv1$; hence $\lambda(\varphi)=0=1/\textsf{R}$.

Assume $|\lambda_j|\ne\lambda(\varphi)$ for some $j$; then we have
$$
\frac{R(\varphi^k)}{\lambda_j^k}=\sum_{i\ne j}\rho_i\left(\frac{\lambda_i}{\lambda_j}\right)^k+\rho_j,\quad
\lim\sum_{i\ne j}\rho_i\left(\frac{\lambda_i}{\lambda_j}\right)^k=\infty.
$$
It follows from the above observations that $1/\textsf{R}=\limsup(\sum_{i\ne j}\rho_i\lambda_i^k)^{1/k}$.
Consequently, we may assume that $R(\varphi^k)=\sum_j\rho_j\lambda_j^k$
with all $|\lambda_j|=\lambda(\varphi)$ and then we have
$$
\frac{1}{\textsf{R}}=\limsup \left(\sum_{|\lambda_j|=\lambda(\varphi)} \rho_j\lambda_j^k\right)^{1/k}.
$$
If $\lambda(\varphi)>1$, then $R(\varphi^k)\to\infty$ and by L'Hopital's rule we obtain
\begin{align*}
\limsup_{k\to\infty}\frac{\log R(\varphi^k)}{k}&=\limsup_{k\to\infty}\frac{\log\left(\sum_j\rho_j\lambda_j^k\right)}{k}
=\log\lambda(\varphi)\ \Rightarrow\ \frac{1}{\textsf{R}}=\lambda(\varphi).
\end{align*}
If $\lambda(\varphi)=1$, then $R(\varphi^k)\le \sum_j|\rho_j|<\infty$ is a bounded sequence
and so it has a convergent subsequence.
If $\limsup R(\varphi^k)=0$, then $R(\varphi^k)=0$ for all $k$ sufficiently large
and so by the same reason as above, $\lambda(\varphi)=0$, a contradiction.
Hence $\limsup R(\varphi^k)$ is a finite nonzero integer and so $1/\textsf{R}=1=\lambda(\varphi)$.

Summing up, we have obtained that

\begin{Thm}\label{Radius1}
Let $\varphi:\Pi\to\Pi$ be an endomorphism on a poly-Bieberbach group $\Pi$
such that all $R(\varphi^k)$ are finite.
Let {\rm $\textsf{R}$} denote the radius of convergence of the Reidemeister zeta function $R_\varphi(z)$ of $\varphi$.
Then $\lambda(\varphi)=0$ or $\lambda(\varphi)\ge1$, and
\begin{align}\label{R1}
\frac{1}{\text{\rm{$\textsf{R}$}}}=\lambda(\varphi).\notag
\end{align}
In particular, {\rm $\textsf{R}>0$}.
\end{Thm}

Recall that
\begin{align*}
&S_\varphi(z)=\sum_{i=1}^{r(\varphi)}\frac{\rho_i\lambda_i}{1-\lambda_iz},\\
&R_\varphi(z)=\prod_{i=1}^{r(\varphi)}(1-\lambda_iz)^{-\rho_i}
=\frac{\prod_{\rho_j<0}(1-\lambda_jz)^{-\rho_j}}{\prod_{\rho_i>0}(1-\lambda_iz)^{\rho_i}}.
\end{align*}
These show that all of the $1/\lambda_i$ are the poles of $S_\varphi(z)$,
whereas the $1/\lambda_i$ with corresponding $\rho_i>0$ are the poles of $R_\varphi(z)$.
The radius of convergence of a power series centered at a point $a$ is equal to the distance from $a$
to the nearest point where the power series cannot be defined in a way that makes it holomorphic.
Hence the radius of convergence of $S_\varphi(z)$ is {$1/\lambda(\varphi)$}
and the radius of convergence of $R_\varphi(z)$ is {$1/\max\{|\lambda_i|\mid \rho_i>0\}$}.
In particular, we have shown that
$$
\lambda(\varphi)=\max\{|\lambda_i|\mid i=1,\cdots,r(\varphi)\}=\max\{|\lambda_i|\mid \rho_i>0\}.
$$

\begin{Thm}\label{Radius2}
Let $\varphi:\Pi\to\Pi$ be an endomorphism on a poly-Bieberbach group $\Pi$ of $S$
such that all $R(\varphi^k)$ are finite.
Let {\rm $\textsf{R}$} denote the radius of convergence of the Reidemeister zeta function of $\varphi$.
If $\varphi$ is the semi-conjugate by an affine map $(d,D)$ on $S$ and if $D_*$ has no eigenvalue $1$, then
$$
\frac{1}{\text{\rm{$\textsf{R}$}}}=\sp\left(\bigwedge D_*\right)=\lambda(\varphi).
$$
\end{Thm}

\begin{proof}
Recall that $R_\varphi(z)=R_f(z)$ and the radius $\textsf{R}$ of convergence of $R_f(z)$
satisfies $1/\textsf{R}=\sp\left(\bigwedge D_*\right)$ by \cite[Theorem~3.4]{FL2}. With Theorem~\ref{Radius1},
we obtain the required assertion.
\end{proof}

We recall that the asymptotic Reidemeister number of $\varphi$ is defined to be
$$
R^\infty(\varphi):=\max\left\{1,\limsup_{k\to\infty}R(\varphi^k)^{1/k}\right\}.
$$
We also recall that the most widely used measure for the complexity of a dynamical system
is the topological entropy $h(f)$. A basic relation between these two numbers
is $h(f)\ge \log N^\infty(f)$, which was found by Ivanov in \cite{I}.
There is a conjectural inequality $h(f)\ge \log(\sp(f))$ raised by Shub \cite{Shub}.
This conjecture was proven for all maps on infra-solvmanifolds of type $\R$, see \cite{mp,mp-a} and \cite{FL}.
Consider a continuous map $f$ on a compact connected manifold $M$, and
consider a homomorphism $\varphi$ induced by $f$ of
the group $\Pi$ of covering transformations on the universal cover of $M$.
Since $M$ is compact, $\Pi$ is finitely generated. Let $T=\{\tau_1,\cdots,\tau_n\}$ be a set of generators for $\Pi$.
For any $\gamma\in\Pi$, let $L(\gamma,T)$ be the length of the shortest word in the letters $T\cup T^{-1}$
which represents $\gamma$.
For each $k>0$, we put
$$
L_k(\varphi,T)=\max \left\{L(\varphi^k(\tau_i),T)\mid i=1,\cdots,n\right\}.
$$
Then the {\bf algebraic entropy} $h_\alg(f)=h_\alg(\varphi)$ of $f$ or $\varphi$ is defined as follows:
$$
h_\alg(f)=\lim_{k\to\infty}\frac{1}{k}\log L_k(\varphi,T).
$$
The algebraic entropy of $f$ is well-defined, i.e., independent of
the choices of a set $T$ of generators for $\Pi$ and a homomorphism $\varphi$ induced by $f$ (\cite[p.~\!114]{KH}).
We refer to \cite{KH} for the background.
We recall that R. Bowen in \cite{B78} and A. Katok in \cite{K}, among others, have proved that
the topological entropy $h(f)$ of $f$ is at least as large as the algebraic entropy
$h_\alg(\varphi)$ of $\varphi$.
Furthermore, for any inner automorphism $\tau_{\gamma_0}$ by $\gamma_0$, we have $h_\alg(\tau_{\gamma_0}\varphi)=h_\alg(\varphi)$ (\cite[Proposition~3.1.10]{KH}).
Now we can make a statement about the relations between $R^\infty(\varphi)$, $\lambda(\varphi)$, $h(f)$ and $h_\alg(\varphi)$.

\begin{Cor}\label{alg entropy}
Let $\varphi:\Pi\to\Pi$ be a homomorphism on a poly-Bieberbach group $\Pi$ of $S$ and all $R(\varphi^k)$ are finite. Let $(d,D)$ be an affine map on $S$ such that $\varphi(\alpha)\circ(d,D)=(d,D)\circ\alpha$ for all $\alpha\in\Pi$. Let $\bar{f}$ be the map on $\Pi\bs{S}$ induced by $(d,D)$ and let $f$ be any map on $\Pi\bs{S}$ which is homotopic to $\bar{f}$. Then
\begin{align*}
&R^\infty(\varphi)=\sp\left(\bigwedge D_*\right)=\lambda(\varphi),\\
&h_\alg(\varphi)=h_\alg(\bar{f})=h_\alg(f)\le h(\bar{f})=\log R^\infty(\varphi)\le h(f),
\end{align*}
provided that $1$ is not an eigenvalue of $D_*$.
\end{Cor}

\begin{proof}
From \cite[Theorem~4.3]{FL} and Theorem~\ref{Radius2}, we obtain the first assertion, $R^\infty(\varphi)=\sp\left(\bigwedge D_*\right)=\lambda(\varphi)$. By \cite[Theorem~5.2]{FL}, $h(f)\ge  h(\bar{f})=\log\lambda(\varphi)$ and by the remark mentioned just above, we have that $h(\bar{f})\ge h_\alg(\bar{f})=h_\alg(f)=h_\alg(\varphi)$.
\end{proof}

\begin{Rmk}
The inequality
$$
\log R^\infty(\varphi)\ge h_\alg(\varphi)
$$
in Corollary~\ref{alg entropy} can be regarded as an algebraic
analogue of the Ivanov inequality $h(f)\ge \log N^\infty(f)$.
\end{Rmk}

\section{Asymptotic behavior of the sequence $\{R(\varphi^k)\}$}

In this section, we study the asymptotic behavior of the Reidemeister numbers of iterates of maps
on poly-Bieberbach groups.

\begin{Thm}\label{BaBo2.6}
Let $\varphi:\Pi\to\Pi$ be an endomorphism on a poly-Bieberbach group
such that all $R(\varphi^k)$ are finite. Then one of the following two possibilities holds:
\begin{enumerate}
\item[$(1)$] $\lambda(\varphi)=0$, which occurs if and only if $R_\varphi(z)\equiv1$.
\item[$(2)$] The sequence $\{R(\varphi^k)/\lambda(\varphi)^k\}$ has the same limit points
as a periodic sequence $\{\sum_j\alpha_j\epsilon_j^k\}$
where $\alpha_j\in\bbz,\epsilon_j\in\bbc$ and $\epsilon_j^q=1$ for some $q>0$.
\item[$(3)$] {The set of limit points of the sequence $\{R(\varphi^k)/\lambda(\varphi)^k\}$ contains an interval.}
\end{enumerate}
\end{Thm}

In Theorem~\ref{Radius2}, we showed that if $D_*$ has no eigenvalue $1$
then $\lambda(\varphi)=\sp(\bigwedge D_*)$. In fact, we have the following:

\begin{Lemma}\label{Radius3}
Let $\varphi$ be a homomorphism on a poly-Bieberbach group $\Pi$ of $S$
and let $\varphi$ be the semi-conjugate by an affine map $(d,D)$ on $S$.
If $\lambda(\varphi)\ge1$, then $\lambda(\varphi)=\sp(\bigwedge D_*)$.
\end{Lemma}

It is important to know not only the rate of growth of the sequence $\{R(\varphi^k)\}$
but also the frequency with which the largest Reidemeister number is encountered.
The following theorem shows that this sequence grows relatively densely.
The following are variations of Theorem~2.7, Proposition~2.8 and Corollary~2.9 of \cite{BaBo}.

\begin{Thm}\label{BaBo2.7}
Let $\varphi:\Pi\to\Pi$ be an endomorphism on a poly-Bieberbach group $\Pi$
such that all $R(\varphi^k)$ are finite.
If $\lambda(\varphi)\ge1$, then there exist $\gamma>0$ and a natural number $N$
such that for any $m> N$ there is an $\ell\in\{0,1,\cdots,n(\varphi)-1\}$
such that $R(\varphi^{m+\ell})/\lambda(\varphi)^{m+\ell}>\gamma$.
\end{Thm}

\begin{Prop}
\label{BaBo2.8}
Let $\varphi:\Pi\to\Pi$ be an endomorphism on a poly-Bieberbach group
such that all $R(\varphi^k)$ are finite and   such that $\lambda(\varphi)>1$.
Then for any $\epsilon>0$, there exists $N$ such that
if $R(\varphi^m)/\lambda(\varphi)^m\ge\epsilon$ for $m>N$,
then the Dold multiplicity $I_m(\varphi)$ satisfies
$$
|I_m(\varphi)|\ge\frac{\epsilon}{2}\lambda(\varphi)^m.
$$
\end{Prop}

Theorem~\ref{BaBo2.7} and Proposition~\ref{BaBo2.8} immediately imply the following:
\begin{Cor}\label{BaBo2.9}
Let $\varphi:\Pi\to\Pi$ be an endomorphism on a poly-Bieberbach group such that all $R(\varphi^k)$
are finite and such that $\lambda(\varphi)>1$.
Then there exist $\gamma>0$ and a natural number $N$ such that if $m\ge N$
then there exists $\ell$ with $0\le\ell\le n(\varphi)-1$
such that $|I_{m+\ell}(\varphi)|/\lambda(\varphi)^{m+\ell}\ge \gamma/2$.
In particular $I_{m+\ell}(\varphi)\ne0$ and so $A_{m+\ell}(\varphi)\ne0$.
\end{Cor}

\begin{Rmk}\label{density}
We can state a little bit more about the density of the set
of algebraic periods $\calA(\varphi)=\{m\in\bbn\mid A_m(\varphi)\ne0\}$.
We consider the notion of the {\bf lower density} $\DA(\varphi)$
of the set $\calA(\varphi)\subset\bbn$:
$$
\DA(\varphi)=\liminf_{k\to\infty}\frac{\#(\calA(\varphi)\cap[1,k])}{k}.
$$

{By Corollary~\ref{BaBo2.9}, when $\lambda(\varphi)>1$, we have $\DA(\varphi)\ge 1/n(\varphi)$.
On the other hand, Theorem~\ref{BaBo2.6} implies the following: If $\lambda(\varphi)=0$ then $R(\varphi^k)=0$ for all $k>0$; by Theorem~\ref{Alg} $A_k(\varphi)=0$ for all $k>0$, hence $DA(\varphi)=\emptyset$.
Consider Case (2) of Theorem~\ref{BaBo2.6}, that is, the sequence $\{R(\varphi^k)/\lambda(\varphi)^k\}$ has the same limit points as the periodic sequence $\{\sum_{j=1}^{n(\varphi)}\rho_je^{2i\pi(k\theta_j)}\}$
of period $q=\lcm(q_1,\cdots,q_{n(\varphi)})$.
By Theorem~\ref{BaBo2.7}, we have $\DA(\varphi)\ge 1/q$.
Finally consider Case (3). Then the sequence $\{R(\varphi^k)/\lambda(\varphi)^k\}$ asymptotically
has a subsequence $\{\sum_{j\in\calS}\rho_je^{2i\pi(k\theta_j)}\}$ where
$\calS = \{j_1,\cdots,j_s\}$ and $\{\theta_{j_1},\cdots,\theta_{j_s},1\}$ is linearly independent over the integers. Therefore by \cite[Theorem 6, p.~\!91]{Ch}, the sequence $(k\theta_{j_1},\cdots,k\theta_{j_s})$ is uniformly distributed. It follows that $\DA(\varphi) = 1$.}
\end{Rmk}

\section{Periodic $[\varphi]$-orbits}

In this section, we shall give an estimate from below the number of {\bf periodic $[\varphi]$-orbits}
of an endomorphism $\varphi$ on a poly-Bieberbach group based on facts discussed in Section~\ref{poly-B}.
We keep in mind that all periodic classes are essential, see Proposition~\ref{ess}.

We denote by $\calO([\varphi],k)$ the set of all (essential) periodic orbits of $[\varphi]$ with length $\le k$.
Thus
$$
\calO([\varphi],k)=\{\langle[\alpha]^m\rangle\mid \alpha\in\Pi, m\le k\}.
$$
Recalling from Section~\ref{poly-B} that $\calO([\varphi],k)=\calO(f,k)$,
we can restate Theorem~\ref{N_BaBo4.2} as follows:

\begin{Thm}\label{BaBo4.2}
Let $\varphi:\Pi\to\Pi$ be an endomorphism on a poly-Bieberbach group
such that all $R(\varphi^k)$ are finite.
Suppose that the sequence $R(\varphi^k)$ is unbounded.
Then there exists a natural number $N_0$ such that
$$
k\ge N_0\Longrightarrow \#\calO([\varphi],k)\ge \frac{k-N_0}{r(\varphi)}.
$$
\end{Thm}

\begin{Prop}\label{epp}
Let $\varphi:\Pi\to\Pi$ be an endomorphism on a poly-Bieberbach group
such that all $R(\varphi^k)$ are finite.
For every $k>0$, we have
\begin{align*}
\# \calI\calR(\varphi^k)=\sum_{d\mid k}\mu\!\left(\frac{k}{d}\right) R(\varphi^d)=I_k(\varphi).
\end{align*}
\end{Prop}

\begin{proof}
We apply the M\"{o}bius inversion formula to the identity
$$
R(\varphi^k)=\sum_{d\mid k}\#\calI\calR(\varphi^d)
$$
in Section~\ref{poly-B} to obtain
$\#\calI\calR(\varphi^k)=\sum_{d\mid k}\mu\!\left(\frac{k}{d}\right)R(\varphi^d)$,
which is exactly the Dold multiplicity $I_k(\varphi)$.
\end{proof}

\begin{Def}
When all $R(\varphi^k)$ are finite, we consider the mod $2$ reduction of
the Reidemeister number $R(\varphi^k)$ of $f^k$, written $R^{(2)}(\varphi^k)$.
A positive integer $k$ is a {\bf $R^{(2)}$-period}
of $\varphi$ if $R^{(2)}(\varphi^{k+i})=R^{(2)}(\varphi^i)$
for all $i\ge1$. We denote the minimal $R^{(2)}$-period of $\varphi$ by $\alpha^{(2)}(\varphi)$.
\end{Def}

\begin{Prop}[{\cite[Proposition~1]{Matsuoka}}]
Let $p$ be a prime number and let $A$ be a square matrix with entries in the field $\bbf_p$.
Then there exists $k$ with $(p,k)=1$ such that
$$
\tr A^{k+i}=\tr A^i
$$
for all $i\ge1$.
\end{Prop}

Recalling \eqref{N2}: $R(\varphi^k)=N(f^k)=\tr M_+^k-\tr M_-^k=\tr(M_+\oplus -M_-)^k$,
we can see easily that the minimal $R^{(2)}$-period $\alpha^{(2)}(\varphi)$ always exists
and must be an odd number.

Now we obtain a result which resembles \cite[Theorem~2]{Matsuoka}.
\begin{Thm}\label{3.3.11}
Let $\varphi:\Pi\to\Pi$ be an endomorphism on a poly-Bieberbach group
such that all $R(\varphi^k)$ are finite.
Let $k>0$ be an odd number. Suppose that $\alpha^{(2)}(\varphi)^2\mid k$ or $p\mid k$
where $p$ is a prime such that $p\equiv 2^i\mod{\alpha^{(2)}(\varphi)}$ for some $i\ge0$.
Then
$$
\frac{\NP_k(\varphi)}{k}=\frac{\#\calI\calR(\varphi^k)}{k}
$$
is even.
\end{Thm}

\section{Heights of $\varphi$}

In this section, we study (homotopy) heights $\calH\calI(\varphi)=\calH(\varphi)$
of Reidemeister classes of endomorphisms $\varphi$ on poly-Bieberbach groups.
We wish to determine the set $\calH(\varphi)$ of all heights only from the knowledge of the sequence $\{R(\varphi^k)\}$.
Recalling that when all $R(\varphi^k)$ are finite, $R(\varphi^k)=\sum_{i=1}^{r(\varphi)}\rho_i\lambda_i^k$
and $\lambda(\varphi)=\max\{|\lambda_i|\mid i=1,\cdots,r(\varphi)\}$,
we define
$$
R^{|\lambda|}(\varphi^k)=\sum_{|\lambda_i|=|\lambda|}\rho_i\lambda_i^k,\quad
\tilde{R}^{|\lambda|}(\varphi^k)=\frac{1}{|\lambda|^k}R^{|\lambda|}(\varphi^k).
$$

\begin{Lemma}\label{3.2.47}
When all $R(\varphi^k)$ are finite, if $\lambda(\varphi)\ge1$, then we have
$$
\limsup_{k\to\infty}\frac{R(\varphi^k)}{\lambda(\varphi)^k}
=\limsup_{k\to\infty}|\tilde{R}^{\lambda(\varphi)}(\varphi^k)|.
$$
\end{Lemma}

\begin{proof}
We have
\begin{align*}
\frac{R(\varphi^k)}{\lambda(\varphi)^k}=\tilde{R}^{\lambda(\varphi)}(f^k)
+\frac{1}{\lambda(\varphi)^k}\sum_{|\lambda_i|<\lambda(\varphi)} \rho_i\lambda_i^k.
\end{align*}
Since for $|\lambda_i|<\lambda(\varphi)$, $\lim\lambda_i^k/\lambda(\varphi)^k=0$,
it follows that the proof is completed.
\end{proof}

\begin{Thm}\label{3.2.48}
Let $\varphi:\Pi\to\Pi$ be an endomorphism on a poly-Bieberbach group
such that all $R(\varphi^k)$ are finite.
Suppose that the sequence {$R(\varphi^k)/\lambda(\varphi)^k$ is asymptotically periodic.}
Then there exist an integer $m>0$ and an infinite sequence $\{p_i\}$ of primes
such that $\{mp_i\}\subset\calH(\varphi)$.
\end{Thm}

\begin{proof}
Since the sequence $R(\varphi^k)$ is unbounded, by Theorem~\ref{BaBo2.6},
there exists $q$ such that all $\lambda_i/|\lambda_i|$ with $|\lambda_i|=\lambda(\varphi)$
are roots of unity of degree $q$, and the sequence $\{\tilde{R}^{\lambda(\varphi)}(\varphi^k)\}$
is periodic and nonzero, because $\limsup_{k\to\infty}|\tilde{R}^{\lambda(\varphi)}(\varphi^k)|>0$
by Lemma~\ref{3.2.47}. Consequently, there exists $m$ with $1\le m\le q$
such that $\tilde{R}^{\lambda(\varphi)}(\varphi^m)\ne0$.

Let $\psi=\varphi^m$. Then $\lambda(\psi)=\lambda(\varphi^m)=\lambda(\varphi)^m\ge1$.
The periodicity $\tilde{R}^{\lambda(\varphi)}(\varphi^{m+\ell q})=\tilde{R}^{\lambda(\varphi)}(\varphi^m)$
implies that $\tilde{R}^{\lambda(\psi)}(\psi^{1+\ell q})=\tilde{R}^{\lambda(\psi)}(\psi)$ for all $\ell>0$.
By Lemma~\ref{3.2.47} or Theorem~\ref{BaBo2.6}, we can see that there exists $\gamma>0$
such that $R(\psi^{1+\ell q})\ge \gamma\lambda(\psi)^{1+\ell q}>0$ for all $\ell$ sufficiently large.
From Proposition~\ref{BaBo2.8} it follows that the Dold multiplicity $I_{1+\ell q}(\psi)$
satisfies $|I_{1+\ell q}(\psi)|\ge (\gamma/2)\lambda(\psi)^{1+\ell q}$ when $\ell$ is sufficiently large.

According to Dirichlet prime number theorem, since $(1,q)=1$,
there are infinitely many primes $p$ of the form $1+\ell q$.
Consider all primes $p_i$ satisfying $|I_{p_i}(\psi)|\ge (\gamma/2)\lambda(\psi)^{p_i}$.

By Proposition~\ref{epp}, $\# \calI\calR(\psi^{p_i})=I_{p_i}(\psi)>0$,
each $p_i$ is the height of some (essential) Reidemeister class $[\alpha]^{p_i}\in\calR[\psi^{p_i}]$.
That is, $[\alpha]^{p_i}$ is an irreducible Reidemeister class of $\psi^{p_i}$.
Consider the Reidemeister class $[\alpha]^{mp_i}$ determined by $\alpha$ of $\varphi^{mp_i}$.
Let $d_i$ be the depth of the Reidemeister class $[\alpha]^{mp_i}\in\calR[\varphi^{mp_i}]$.
Then $d_i= m_ip_i$ for some $m_i\mid m$ and so there is an irreducible Reidemeister class
$[\beta]^{d_i}\in\calR[\varphi^{d_i}]$ which is boosted to $[\alpha]^{mp_i}$.
This means that $d_i$ is the height of $[\beta]^{d_i}$.
Choose a subsequence $\{m_{i_k}\}$ of the sequence $\{m_i\}$ bounded by $m$ which is constant,
say $m_0$. Consequently, the infinite sequence $\{m_0p_{i_k}\}$ consists
of heights of $\varphi$, or $\{m_0p_i\}\subset\calH(\varphi)$.
\end{proof}

{In the proof of Theorem~\ref{3.2.48}, we have shown the following,
which proves that the algebraic period is a (homotopy) height when it is a prime number.}
\begin{Cor}\label{3.2.50}
Let $\varphi:\Pi\to\Pi$ be an endomorphism on a poly-Bieberbach group
such that all $R(\varphi^k)$ are finite. For all primes $p$,
if $A_p(\varphi)\ne0$ then {$p\in\calH(\varphi)$}.
\end{Cor}

\begin{Cor}\label{3.2.51}
Let $\varphi:\Pi\to\Pi$ be an endomorphism on a poly-Bieberbach group
such that all $R(\varphi^k)$ are finite.
If the sequence $\{R(\varphi^k)\}$ is strictly monotone increasing,
then there exists $N$ such that the set $\calH(\varphi)$ contains
all primes larger than $N$.
\end{Cor}

\begin{proof}
{By the assumption, we have $\lambda(\varphi)>1$.}
Thus by Theorem~\ref{BaBo2.7}, there exist $\gamma>0$ and $N$ such that
if $k>N$ then there exists $\ell=\ell(k)<r(\varphi)$ such that $R(\varphi^{k-\ell})/\lambda(\varphi)^{k-\ell}>\gamma$.
Then for all $k>N$, the monotonicity gives
\begin{align*}
\frac{R(\varphi^k)}{\lambda(\varphi)^k}\ge \frac{R(\varphi^{k-\ell})}
{\lambda(\varphi)^k}=\frac{R(\varphi^{k-\ell})}
{\lambda(\varphi)^{k-\ell}\lambda(\varphi)^\ell}
\ge\frac{\gamma}{\lambda(\varphi)^\ell}\ge\frac{\gamma}{\lambda(\varphi)^{r(\varphi)}}.
\end{align*}
Applying Proposition~\ref{BaBo2.8} with $\epsilon=\gamma/\lambda(\varphi)^{r(\varphi)}$,
we see that $I_k(\varphi)\ne0$ and so $A_k(\varphi)\ne0$ for all $k$ sufficiently large.
Now our assertion follows from Corollary~\ref{3.2.50}.
\end{proof}

\begin{Example}
There are examples of groups and endomorphisms satisfying the conditions of the above Corollary.
The simplest one is the endomorphism $\varphi:\bbz\to\bbz$ such that $\varphi(1)=d$.
Then $\varphi^k(1)=d^k$ and so $R(\varphi^k)=|1-d^k|$ for all $k>0$.
When $d\ge2$, it is easy to see that all $R(\varphi^k)$ are finite and the sequence $\{R(\varphi^k)\}$ is strictly increasing.

By a direct computation, we can show that
$$
\calH(\varphi)=\begin{cases}
\{1\}&\text{if $d=0$ or $-1$}\\
\emptyset&\text{if $d=1$}\\
\bbn-\{2\}&\text{if $d=-2$}\\
\bbn&\text{if $d\ge2$ or $d\le-3$.}
\end{cases}
$$
In fact, when $d=1$, all the Reidemeister classes are inessential and hence by definition $\calH(\varphi)=\emptyset$.
For another instance, consider the case $d=-2$.
For any $k\ge1$, $\varphi^k(1)=(-2)^k$ and so the Reidemeister class
$[n]^k\in\calR[\varphi^k]$ is
$$
[n]^k=\{m+n-(-2)^km\mid m\in\bbz\}=n+(1-(-2)^k)\bbz.
$$
Since $\iota_{1,2}([n]^1)=[n+\varphi(n)]^2=[n-2n]^2=[-n]^2$, it follows that $2\notin\calH(\varphi)$.
Next, we remark that
\begin{align*}
\iota_{k,\ell}([n]^k)&=[n+\varphi^k(n)+\cdots+\varphi^{\ell-k}(n)]^\ell\\
&=[(1+(-2)^k+\cdots+(-2)^{\ell-k})n]^\ell \\
&=\left[\frac{1-(-2)^\ell}{1-(-2)^k}n\right]^\ell.
\end{align*}
For $0\le n<|1-(-2)^k|$, when $\ell\ne2$ we see that
$$
\left|\frac{1-(-2)^\ell}{1-(-2)^k}\right|\ne1,\quad
0\le \left|\frac{1-(-2)^\ell}{1-(-2)^k}\right|n<|1-(-2)^\ell|.
$$
This implies that if $\ell\ne2$ then $\ell\in \calH(\varphi)$.
Consequently, $\calH(\varphi)=\bbn-\{2\}$.
Now the remaining cases can be treated in a similar way and we omit a detailed computation.

\end{Example}

An endomorphism $\varphi:\Pi\to\Pi$ is {\bf essentially reducible}
if any Reidemeister class of $\varphi^k$ being boosted
to an essential Reidemeister of $\varphi^{kn}$ is essential,
for any positive integers $k$ and $n$.
The group $\Pi$ is {\bf essentially reducible} if every endomorphism on $\Pi$ is essentially reducible.

\begin{Lemma}[{\cite[Lemma~6.7]{FL2}}]\label{er}
Every poly-Bieberbach group is essentially reducible.
\end{Lemma}

This means that for any $n$, if $[\alpha]^n$ is essential
and if $\iota_{m,n}([\beta]^m)=[\alpha]^n$ then $[\beta]^m$ is essential.

\begin{Lemma}[{\cite[Proposition~2.2]{ABLSS}}]\label{Hal}
Let $\varphi:\Pi\to\Pi$ be an endomorphism such that all $R(\varphi^k)$ are finite. If
$$
\sum_{\frac{m}{k}:\text{ {\rm prime}}} R(\varphi^k)< R(\varphi^m),
$$
then $\varphi$ has a periodic Reidemeister class with height $m$, i.e., $m\in\calH(\varphi)$.
\end{Lemma}

\begin{proof}
Let $r=R(\varphi^m)$ and let $[\alpha_1]^m,\cdots,[\alpha_r]^m$ be the Reidemeister classes of $\varphi^m$.
If some $[\alpha_j]^m$ is irreducible, then we are done. So assume no $[\alpha_j]^n$ is irreducible.
Then, for each $j$, there is a $k_j$ so that $m/k_j$ is prime
and $[\alpha_j]^m$ is reducible to $[\beta_j]^{k_j}\in\calR[\varphi^{k_j}]$.
But this shows that $R(\varphi^m)\le\sum_{\frac{m}{k}: \text{prime}}R(\varphi^k)$, a contradiction.
\end{proof}

We can not only extend but also strengthen Corollary~\ref{3.2.51} as follows:
\begin{Prop}\label{prime power}
Let $\varphi:\Pi\to\Pi$ be an endomorphism on a poly-Bieberbach group
such that all $R(\varphi^k)$ are finite.
Suppose that the sequence $\{R(\varphi^k)\}$ is strictly monotone increasing. Then:
\begin{enumerate}
\item[$(1)$] All primes belong to $\calH(\varphi)$.
\item[$(2)$] There exists $N$ such that if $p$ is a prime $>N$
then $\{p^n\mid n\in\bbn\}\subset\calH(\varphi)$.
\end{enumerate}
\end{Prop}

\begin{proof}
Observe that for any prime $p$
$$
R(\varphi^p)-\sum_{\frac{p}{k}:\text{ {\rm prime}}} R(\varphi^k)
=R(\varphi^p)-R(\varphi)=I_p(\varphi).
$$
The strict monotonicity implies $A_p(\varphi)=pI_p(\varphi)>0$,
and hence $p\in\calH(\varphi)$, which proves (1).

Under the same assumption, we have shown in the proof of Corollary~\ref{3.2.51}
that there exists $N$ such that $k>N\Rightarrow I_k(\varphi)>0$.
Let $p$ be a prime $>N$ and $n\in\bbn$. Then
\begin{align*}
&R(\varphi^{p^n})-\sum_{\frac{p^n}{k}:\text{ {\rm prime}}} R(\varphi^k)
=\sum_{i=0}^n I_{p^{i}}(\varphi)-R(\varphi^{p^{n-1}})=I_{p^n}(\varphi)>0.
\end{align*}
By Lemma~\ref{Hal}, we have $p^n\in\calH(\varphi)$, which proves (2).
\end{proof}

In Remark~\ref{density}, we made a statement about the lower density $\DA(\varphi)$
of the set of algebraic periods $\calA(\varphi)=\{m\in\bbn\mid A_m(\varphi)\ne0\}$.
We can consider as well the lower density of the set $\calH(\varphi)$ of heights,
see also \cite{JLZ}, \cite{Zhao et al} and \cite{FL2}:
\begin{align*}
&\DH(\varphi)=\liminf_{k\to\infty}\frac{\#(\calH(\varphi)\cap[1,k])}{k}.
\end{align*}
Since $I_k(\varphi)=\#\calI\calR(\varphi^k)$ by Proposition~\ref{epp},
it follows that $\calA(\varphi)\subset \calH\calI(\varphi)=\calH(\varphi)$.
Hence we have $\DA(\varphi)\le\DH(\varphi)$.

\begin{Cor}\label{cofinite}
Let $\varphi:\Pi\to\Pi$ be an endomorphism on a poly-Bieberbach group
such that all $R(\varphi^k)$ are finite.
Suppose that the sequence $\{R(\varphi^k)\}$ is strictly monotone increasing.
Then $\calH(\varphi)$ is cofinite and $\DA(\varphi)=\DH(\varphi)=1$.
\end{Cor}

\begin{proof}
Under the same assumption, we have shown in the proof of Corollary~\ref{3.2.51}
that there exists $N$ such that if $k>N$ then $I_k(\varphi)>0$.
This means $\calI\calR(\varphi^k)$ is nonempty by Proposition~\ref{epp}
and hence $k\in\calH(\varphi)$.
\end{proof}

Let $\varphi:\Pi\to\Pi$ be an endomorphism on a poly-Bieberbach group $\Pi$ of $S$
such that all $R(\varphi^k)$ are finite.
When $\varphi$ is the semi-conjugate by an affine map $(d,D)$ on $S$,
we say that $\varphi$ is {\bf expanding} if all the eigenvalues of $D_*$ have modulus $>1$.

Now we can prove the main result of \cite{LZ-jmsj}.
\begin{Cor}[{\cite[Theorem~4.6]{LL-JGP}, \cite[Theorem~3.2]{LZ-jmsj}}]
Let $\varphi$ be an expanding endomorphism on an almost Bieberbach group.
Then $\HPer(\varphi)$ is cofinite.
\end{Cor}

\begin{proof}
Since $\varphi$ is expanding, we have that $\lambda(\varphi)=\sp(\bigwedge D_*)>1$.
For any $k>0$, we can write $R(\varphi^k)=\Gamma_k+\Omega_k$, where
$$
\Gamma_k=\lambda(\varphi)^k\left(\sum_{j=1}^{n(\varphi)}\rho_je^{2i\pi(k\theta_j)}\right),\quad
\Omega_k=\sum_{i=n(\varphi)+1}^{r(\varphi)}\rho_i\lambda_i^k\
\text{ with $|\lambda_i|<\lambda(\varphi)$.}
$$
Here $\Omega_k\to0$ and $\Gamma_k\to\infty$ as $k\to\infty$.
This implies that $R(\varphi^k)$ is eventually strictly monotone increasing.
We can use Corollary~\ref{3.2.51} and then Corollary~\ref{cofinite} to conclude the assertion.
\end{proof}

\end{document}